\documentclass[11p,reqno]{amsart}

\usepackage[T1]{fontenc}
\usepackage{graphicx}
\usepackage{amssymb,amsthm,amsmath,mathrsfs,bm,braket,marginnote}
\usepackage{enumerate}
\usepackage{appendix}
\usepackage[colorlinks=true, pdfstartview=FitV, linkcolor=blue, citecolor=blue, urlcolor=blue]{hyperref}
\usepackage{multirow}
\usepackage{float} 
\usepackage{graphicx} 
\usepackage{subfigure} 

\usepackage{pgf}
\usepackage{pgfplots}
\usepackage{tikz}
\usetikzlibrary{arrows,calc}
\usepackage{verbatim}
\usetikzlibrary{decorations.pathreplacing,decorations.pathmorphing}
\usepackage[numbers,sort&compress]{natbib}
\usepackage{dsfont}

\numberwithin{equation}{section}
\newtheorem{theorem}{Theorem}[section]

\newtheorem{definition}[theorem]{Definition}
\newtheorem{remark}[theorem]{Remark}

\newtheorem{proposition}[theorem]{Proposition}

\reversemarginpar
\newcommand{\mau}{\mathbf{u}}
\newcommand{\mav}{\mathbf{v}}
\newcommand{\gf}{\operatorname{GF}}
\newcommand{\pf}{\operatorname{Pf}}

\newcommand{\sgn}{\text{sgn}}

\newcommand{\puv}{\mathscr{P}(\mathbf{u}; \mathbf{v})}
\newcommand{\Mp}{\mathscr{P}}
\newcommand{\pouv}{\mathscr{P}_0(\mathbf{u}; \mathbf{v})}

\begin{document}
	\title[Combinatoric explanations for Block Pfaffians]{Non-intersecting path explanation for block Pfaffians and applications into skew-orthogonal polynomials}
	
	\date{}
	
	\author{Zong-Jun Yao}
	\address{Department of Mathematics, Sichuan University, Chengdu, 610064, China}
	\email{}
	
	\author{Shi-Hao Li}
	\address{Department of Mathematics, Sichuan University, Chengdu, 610064, China}
	\email{shihao.li@scu.edu.cn}

	\begin{abstract}
		 In this paper, we mainly consider a combinatoric explanation for block Pfaffians in terms of non-intersecting paths, as a generalization of results obtained by Stembridge. As applications, we demonstrate how are generating functions of non-intersecting paths related to skew orthogonal polynomials and their deformations, including a new concept called multiple partial-skew orthogonal polynomials.	
	\end{abstract}
	
	\dedicatory{}
	

	\maketitle

	\section{Introduction}
	Enumerative combinatorics is always playing a fundamental role in mathematics and physics. Among these combinatorial objects, the theory of non-intersecting paths is an extremely useful  theory with applications in counting tilings, plane partitions, and tableaux. In particular, the famous Lindstr\"om-Gessel-Vionnet (LGV) lemma states that the generating function of non-intersecting paths between vertex sets $\mau\to\mav$, where $\mau$ and $\mav$ have the same number of vertices, could be written as a determinant \cite{lindstorm73,gessel85}. This result has many applications such as non-intersecting random walks \cite{johansson02,johansson13}, domino tiling models \cite{duits21}, dimer models \cite{ciucu11}, and so on. It turns out that the generating functions of non-intersecting paths are usually expressed in terms of structured determinants such as Hankel or Toeplitz determinants. Moreover, it has been shown in \cite{nakamura04,chen15} that the Hankel determinant of some special combinatorial numbers could be evaluated by using orthogonal polynomials, continued fraction and discrete integrable systems. Therefore, there is a unified relation between LGV lemma, integrable models, probability models, and orthogonal polynomials. 
	
	In \cite{stembridge90}, the author developed another algebraic tool---Pfaffian, to describe non-intersecting paths between a set of vertices to an interval. This combinatorial description for Pfaffian was also applied to count skew Young tableaux and plane partitions \cite{okada89,stembridge90}. A minor summation formula for Pfaffian \cite{ishiwaka05} was given by using similar combinatorial technique, called \emph{the path switching involution} in \cite{krattenthaler17}. The minor summation formula is as an analogy of Cauchy-Binet formula for determinants, and it was applied to evaluate a certain Catalan-Hankel type Pfaffian
	\begin{align}\label{catalan}
		\pf\left(
		(j-i)\mu_{i+j+k}
		\right)_{i,j=0}^{2n-1},
	\end{align}
	where $\{\mu_{i}\}_{i\in\mathbb{N}}$ are moments sequence related to little $q$-Jacobi polynomials \cite{ishiwaka13}.  In fact, Pfaffian expression \eqref{catalan} is an important quantity in random matrix theory, especially for the characterization of symplectic invariant ensemble \cite{mehta00}.
	By realizing that \eqref{catalan} is a normalization factor for certain skew-orthogonal polynomials, various Catalan-Hankel type Pfaffians were computed by using skew-orthogonal polynomials under Askey-Wilson scheme in \cite{shen21}.  
	
	Inspired by these results, we attempt to give a unified frame between non-intersecting paths, Pfaffian formulas, and skew-orthogonal polynomials. It was known from \cite{stembridge90} that the generating function of non-intersecting paths from $\mau\to I$ could be written as a Pfaffian, namely, 
		\begin{align}\label{evenorder}
		\gf[\mathscr{P}_0(\mathbf{u};I)]=\pf\left[
		Q_I(u^{(i)},u^{(j)})
		\right]_{i,j=1}^{2n},
	\end{align}
	where $\mathscr{P}_0(\mathbf{u};I)$ stands for all non-intersecting paths from an even-numbered vertex set $\mathbf{u}=(u^{(1)},\cdots,u^{(2n)})$ to an interval $I$ and $Q_I(u^{(i)},u^{(j)})$ are some computable weights. Different from the Pfaffian formula considered in \eqref{catalan}, we show that this generating function is related to another type of skew-orthogonal polynomials, which are related to orthogonal invariant ensemble in random matrix theory.
	Besides, if the vertex set $\mathbf{u}=(u^{(1)},\cdots,u^{(2n+1)})$, then its corresponding generating function could be expressed as an augmented Pfaffian
	\begin{align*}
		\gf[\mathscr{P}_0(\mathbf{u};I)]=\pf\left[
		\begin{array}{cc}
			Q_I(u^{(i)},u^{(j)}) &Q_I(u^{(i)})\\
			-Q_I(u^{(j)})&0
		\end{array}
		\right]_{1\leq i,j\leq 2n+1},
	\end{align*}
	where $Q_I(u^{(j)})$ is a weight function related to $u^{(j)}$ only.  
	Since skew-orthogonal polynomials are only related to the even-ordered Pfaffian formula \eqref{evenorder}, we demonstrate that this augmented Pfaffian should be related to partial-skew-orthogonal polynomials proposed in \cite{chang18}.
	
	In recent years, matrix-valued orthogonal polynomials play an important role in probability and combinatoric models such as hexagon tilings \cite{groot21} and periodic Aztec diamond models \cite{duits21}. Despite of block determinants, Pfaffians of skew symmetric block matrices were also proved to be useful in practice \cite{li23}. It was shown in \cite{stembridge90} that the generating function for $\mathbf{u}\to \mathbf{v}\oplus I$ could be expressed as a block Pfaffian, with one block being all zeros. Here $\mathbf{u}$ and $\mathbf{v}$ are two sets of vertices and $I$ is an interval of vertices. Our first result is to give a non-intersecting path interpretation for block Pfaffians
	\begin{align}\label{general}
		\pf\left[
		\begin{array}{cc}
			A& B\\
			-B^\top&C
		\end{array}
		\right].
	\end{align}
	We show that this Pfaffian could be expressed as the generating function of non-intersecting paths $\mathscr{P}_0(J\oplus\mathbf{u};\mathbf{v}\oplus I)$, where $\mathbf{u}$ and $\mathbf{v}$ are two sets of vertices and $I$ and $J$ are two intervals of vertices. Besides, we show that this generating function is related to a 2-component skew-orthogonal polynomials, which could be used to describe a non-intersecting random walk with two different sources. 
	Moreover, a concept of multiple partial-skew-orthogonal polynomials is proposed if the number of vertices in $\mathbf{u}$ and $\mav$ are odd.
	Furthermore, we extend formula \eqref{general} to a more general case, where we consider the paths from $\{\mau_1,\cdots,\mau_n,I_1,\cdots,I_m\}$ to $\{J_1,\cdots,J_n,\mav_1,\cdots,\mav_m\}$. We show that this generating function could also be written as a block Pfaffian. We remark that in the reference \cite{carrozza18}, a non-intersecting path explanation for Pfaffians \eqref{general} was given, but it was shown by using Grassmann algebra in a cyclic digraph. 
	
	This paper is organized as follows. In Section \ref{sec2}, we recall the LGV lemma and show some applications. We connect the LGV lemma with several different orthogonal polynomials (including orthogonal polynomials on the real line and on the unit circle) and  random walks (including discrete-time random walks and continuous-time random walks). In Section \ref{sec3}, we recall Stembridge's results on the non-intersecting path explanation for Pfaffians. Their connections with skew-orthogonal polynomials and partial-skew-orthogonal polynomials are given. We show that this generating function is related to some random walk models as well. Thus we give a unified frame between non-intersecting paths, Pfaffian formulas and skew-orthogonal polynomials. We generalize Stembridge's result in Section \ref{sec4}, where we consider paths from $J\oplus\mathbf{u}\to\mathbf{v}\oplus I$. Multiple skew-orthogonal polynomials are introduced by following this generating function, with some applications in random walks starting from different sources. Moreover, we introduce a new concept of multiple partial-skew-orthogonal polynomials from this generating function.  In Section \ref{sec5}, non-intersecting paths between several vertex sets and intervals are considered. This is the most general case between vertex sets and intervals.

	\section{Lindstr\"om-Gessel-Vionnet theorem and applications into orthogonal polynomials}\label{sec2}
	
	In this part, we give some brief reviews on the connections between Lindstr\"om-Gessel-Vionnet lemma and orthogonal polynomials. There have been numerous references about this topic, and please refer to \cite{gharakhloo24,krattenthaler17,stanley99,sagan01} and references therein.
	
	\begin{definition}
		Let $V$ be a set of vertices, $E$ be a set of directed edges between vertices in $V$, and $D = (V,E)$ represents an acyclic graph formed by $V$ and $E$. We assume that the weight of each edge should be greater than zero.
	\end{definition}
	
	\begin{definition}[$\mathscr{P}(u;v)$]
		Let $u$ and $v$ be two vertices in $D$. We use $\mathscr{P}(u; v)$ to denote the set of all paths from $u$ to $v$ in the graph $D$.
	\end{definition}
	
	Let $P$ be a path from $u$ to $v$, that is, $P\in\mathscr{P}(u;v)$, then the weight of $P$ denoted by $\omega(P)$, is defined as the product of the weights of all edges on this path. Moreover, 
	the weight between two vertices $u$ and $v$ is defined as $h(u, v) = \sum\limits_{P \in {\mathscr{P}}(u;v)} {\omega (P)} $.
	
	\begin{definition}[$\mathscr{P}(\mathbf{u}; \mathbf{v})$, $\pouv$]
		Let $\mau=(u^{(1)},u^{(2)},\cdots,u^{(r)})$ and $\mav=(v^{(1)},v^{(2)},\cdots,v^{(r)})$ be two sequences of vertices in $V$, which are composed by $r$ points in an ascending order, namely, we assume that $u^{(1)}<\cdots<u^{(r)}$ and $v^{(1)}<\cdots<v^{(r)}$. We denote $\puv$ by the set of all paths from $\mau$ to $\mav$ in the graph $D$ in the order $u^{(1)}\to v^{(1)}$, $u^{(2)}\to v^{(2)}$, ..., and $u^{(r)}\to v^{(r)}$. Especially, we denote $\pouv$ as the set of all non-intersecting paths from $\mau$ to $\mav$.
	\end{definition}
	
	Let $\xi=(P_1,P_2,\cdots,P_r)\in\puv$ be an $r$-path starting from $\mau$ to $\mav$, where $P_i$ is a path from $u_i$ to $v_i$.
	The weight of $\xi$ is defined to be $\omega(\xi)=\omega(P_1,P_2,\cdots,P_n)=\prod_{i=1}^r \omega(P_i)$. Moreover, if we count the weight for all paths from $\mau$ to $\mav$, then we have
	\begin{align*}
		h(\mau,\mav)=\sum_{\xi\in\puv} \omega(\xi).
	\end{align*}
	In general, we refer this weight function as the generating function of all paths from $\mau$ to $\mav$, and denote it as GF$[\puv]$.
	For non-intersecting paths, we have GF$[\pouv]=\sum_{\xi\in\pouv}\omega(\xi)$.

	\begin{definition}[$D$-Compatible]
		If $I$ and $J$ are ordered sets of vertices in a graph $D$, then $I$ is said to be D-compatible with $J$ in the graph if and only if for any $u<u^{\prime}$ in $I$ and any $v>v^{\prime}$ in $J$, every path $P \in \mathscr{P}(u; v)$ intersects with every path $Q \in \mathscr{P}\left(u^{\prime}; v^{\prime}\right)$.
	\end{definition}
	
	Now, we could formally state the LGV lemma.
	
	\begin{theorem}[\cite{lindstorm73,gessel85,stembridge90}]\label{lgv}
		
		Let $\mathbf{u}=\left(u^{(1)}, u^{(2)}, \ldots, u^{(r)}\right)$ and $\mathbf{v}=\left(v^{(1)}, v^{(2)}, \ldots, v^{(r)}\right)$ be two ordered $r$-tuples of vertices in an acyclic directed graph $D$. If $\mathbf{u}$ and $\mathbf{v}$ are D-compatible, then
		\begin{equation}
			\gf\left[\mathscr{P}_0(\mathbf{u}; \mathbf{v})\right]=\operatorname{det}\left[h\left(u^{(i)}, v^{(j)}\right)\right]_{1 \leqslant i, j \leqslant r} .
		\end{equation}
		
	\end{theorem}
	
	In literatures, there have been many applications of Lindst\"orm-Gessel-Vionnet lemma. For example, it has been applied to some tiling problem and random growth models \cite{johansson02,duits21}.
	Originally, non-intersecting paths arose in matroid theory \cite{lindstorm73}, which was later used to count tableaux and plane partitions \cite{gessel85}. In the references mentioned above, there is a specialized binomial determinant related to Toeplitz (respectively Hankel) determinant and orthogonal polynomials on the unit circle (respectively real line). Here we start with a non-intersecting random walk model and demonstrate how LGV lemma works in the random walk models and  orthogonal polynomial theory. 
	
	Let's consider $N$ independent simple random walks $X_1(t),\cdots,X_N(t)$ starting from positions $\mathbf{a}=(\alpha_1,\cdots,\alpha_N)$ at $t=0$, and ending at positions $\mathbf{b}=(\beta_1,\cdots,\beta_N)$ at $t=2T$ with conditions $\alpha_1<\cdots<\alpha_N$ and $\beta_1<\cdots<\beta_N$. There are two different models related to 
	this setting. One is a continuous-time model called non-intersecting Brownian motion with configuration $\mathbb{R}\times\mathbb{R}_{\geq 0}$ \cite{baik13,forrester11}. We first assume  $X(t)=(X_1(t),\cdots,X_N(t))$ to be $N$ independent standard Brownian motions with transition probability
	\begin{align*}
		p(x,y;t)=\frac{1}{\sqrt{2\pi t}}e^{-\frac{(x-y)^2}{2t}},
	\end{align*}
	which indicates the probability from $x$ to $y$ within time $t$. Moreover, let's assume that $D_N:=\{
	x_1<\cdots<x_N
	\}\subset \mathbb{R}^N$ is a admissible configuration.
	In this case, non-intersecting means that $X(t) \in D_N$ for all $ t \in (0,2T)$. According to the Karlin-McGregor's theorem \cite{karlin59}, the conditional probability density function (PDF) for passing through $\mathbf{x}=(x_1,\cdots,x_N)$ at time $t$ and ending at $\beta$, given the starting point $\alpha$ and the total duration of $2T$,  is expressed by
	\begin{align}\label{pdf1}
		\operatorname{det}\left(p\left(\alpha_i, x_j ; t\right)\right)_{i, j=1}^N \operatorname{det}\left(p\left(x_j, \beta_i ; 2T-t\right)\right)_{i, j=1}^N.
	\end{align}
	This result could be also recognized as an application of LGV lemma, where $\operatorname{det}\left(p\left(\alpha_i, x_j ; t\right)\right)_{i, j=1}^N$ could be viewed as the generating function from $\mathbf{a}$ to $\mathbf{x}$, and $\operatorname{det}\left(p\left(x_j, \beta_i ; 2T-t\right)\right)_{i, j=1}^N$ could be viewed as the generating function from $\mathbf{x}$ to $\mathbf{b}$. To calculate the conditional probability of event $\mathcal{N}_0$, we normalize it by dividing a normalization factor for reaching $\mathbf{b}$ from $\mathbf{a}$ within the total time of $2T$. In fact, this normalization factor is given by $\prod_{j=1}^N p(\alpha_j-\beta_j; 2T)$, where $p(x;T)$ is a normal distribution with mean zero and variance $T$. 
	Therefore, the conditional probability of the event $\mathcal{N}_0$ could expressed by
	$$
	\mathbb{P}_{\alpha, \beta}(\mathcal{N}_0) = \frac{\operatorname{det}\left[p(\alpha_j-\beta_k;2T)\right]_{j, k=1}^{N}}{\prod_{j=1}^{N} p(\alpha_j-\beta_j;2T)}.
	$$
	In fact, the PDF \eqref{pdf1} induces a bi-orthogonal system. We demonstrate this result in the following proposition.
	\begin{proposition}\label{prop2.6}
		If we denote $\psi_i(x)=p(\alpha_i,x;t)$ and $\phi_j(x)=p(x,\beta_j;2T-t)$, then we could introduce moments as
		\begin{align*}
			m_{\alpha_i,\beta_j}=\int_{\mathbb{R}}\psi_i(x)\phi_j(x)dx.
		\end{align*}
		Moreover, functions
		\begin{align*}
			P_n(x)=\left|\begin{array}{cccc}
				m_{\alpha_1,\beta_1}&m_{\alpha_1,\beta_2}&\cdots&m_{\alpha_1,\beta_n}\\
				m_{\alpha_2,\beta_1}&m_{\alpha_2,\beta_2}&\cdots&m_{\alpha_2,\beta_n}\\
				\vdots&\vdots&&\vdots\\
				m_{\alpha_{n-1},\beta_1}&m_{\alpha_{n-1},\beta_2}&\cdots&m_{\alpha_{n-1},\beta_n}\\
				\phi_1(x)&\phi_2(x)&\cdots&\phi_n(x)
			\end{array}
			\right|
		\end{align*}
		and 
		\begin{align*}
			Q_n(x)=\left|\begin{array}{cccc}
				m_{\alpha_1,\beta_1}&\cdots&m_{\alpha_1,\beta_{n-1}}&\psi_1(x)\\
				m_{\alpha_2,\beta_1}&\cdots&m_{\alpha_2,\beta_{n-1}}&\psi_2(x)\\
				\vdots&\vdots&&\vdots\\
				m_{\alpha_{n},\beta_1}&\cdots&m_{\alpha_{n},\beta_{n-1}}&\psi_n(x)\\
			\end{array}
			\right|
		\end{align*}
		are orthogonal to each other with the uniform measure. In other words, we have 
		\begin{align*}
			\int_{\mathbb{R}}P_n(x)Q_m(x)dx=H_n\delta_{n,m},
		\end{align*}
		where 
		\begin{align*}
			H_n=\det[p(\alpha_j-\beta_k;2T)]_{j,k=1}^{n-1}\det[p(\alpha_j-\beta_k;2T)]_{j,k=1}^{n}.
		\end{align*} 
	\end{proposition}
	This is a general bi-orthogonal system \cite{borodin99}. If we consider confluent starting and ending point, then we could define multiple orthogonal polynomials and multiple orthogonal polynomials of mixed type \cite{daems07}. To be precise, let's consider a confluent case where $N$ non-intersecting Brownian motions start from $k$ different points $\alpha_j$ with multiplicity $a_j$ for $j=1,\cdots,k$, and end at $l$ different points $\beta_j$ with multiplicity $b_j$ for  $j=1,\cdots,l$, such that $\sum_{j=1}^k a_j=\sum_{j=1}^l b_j=N$. Under this circumstance, we have the following proposition for multiple orthogonal polynomials of mixed type. 
	\begin{proposition}
		Let's denote $\psi_i^{(j)}(x)=x^i\omega_{1,j}(x)$ for $i=0,\cdots,a_j-1$ and $j=1,\cdots,k$ and $\phi_i^{(j)}(x)=x^i\omega_{2,j}(x)$ for $i=0,\cdots,b_j-1$ and $j=1,\cdots,l$, where $\omega_{1,j}(x)=p(\alpha_j,x;t)$ and $\omega_{2,j}(x)=p(x,\beta_j;2T-t)$. We define moments of mixed type by 
		\begin{align*}
			m_{p+q}^{(i,j)}=\int_{\mathbb{R}}x^{p+q}\omega_{1,i}(x)\omega_{2,j}(x)dx,\quad p=0,\cdots,a_i-1,\,q=0,\cdots,b_j-1
		\end{align*}
		for $i=1,\cdots,k,\,j=1,\cdots,l.$ In this case, multiple orthogonal polynomials of mixed type could be written as 
		\begin{align*}
			P_{\vec{a}-e_k,\vec{b}}(x)=\left|\begin{array}{ccc}
				A_{a_1,b_1}^{(1,1)}&\cdots&A_{a_1,b_l}^{(1,l)}\\
				\vdots&&\vdots\\
				A_{a_k-1,b_1}^{(k,1)}&\cdots&A_{a_k-1,b_l}^{(k,l)}\\
				\Phi_1(x)&\cdots&\Phi_l(x)
			\end{array}\right|,\quad Q_{\vec{a},\vec{b}-e_l}(x)=\left|\begin{array}{cccc}
				A_{a_1,b_1}^{(1,1)}&\cdots&A^{(1,l)}_{a_1,b_l-1}&\Psi^\top_1(x)\\
				\vdots&&\vdots&\vdots\\
				A_{a_k,b_1}^{(k,1)}&\cdots&A^{(k,l)}_{a_k,b_l-1}&\Psi^\top_k(x)
			\end{array}
			\right|,
		\end{align*}
		where 
		\begin{align*}
			\Psi_j(x)=(1,x,\cdots,x^{a_j-1})\omega_{1,j}(x),\,
			\Phi_j(x)=(1,x,\cdots,x^{b_j-1})\omega_{2,j}(x), \, A_{a_i,b_j}^{(i,j)}=(m_{p+q}^{(i,j)})_{p=0,\cdots,a_i-1\atop q=0,\cdots,b_j-1}.
		\end{align*}
		Moreover, these polynomials satisfy the following mixed orthogonality
		\begin{align*}
			&\int_{\mathbb{R}}P_{\vec{a},\vec{b}}(x)x^p\omega_{1,j}(x)dx=0,\quad p=0,\cdots,a_j-1,\,j=1\cdots,k,\\
			&\int_{\mathbb{R}}Q_{\vec{a},\vec{b}}(x)x^p\omega_{2,j}(x)dx=0,\quad p=0,\cdots,b_j-1,\,j=1\cdots,l.
		\end{align*}
	\end{proposition}

	Another application is to consider a discrete-time random walk, defined in the configuration space $\mathbb{Z}\times \mathbb{Z}_{\geq 0}$. At each discrete time step, the simple random walk increase or decrease by one step with equal probability. If we denote $\alpha_j=2x_j$ and $\beta_j=2y_j$ for all $j=1,2,\cdots,N$, and assume that all paths don't intersect, then the number of all non-intersecting paths is given by a binomial determinant \cite{krattenthaler00,forrester02}
	\begin{align*}
		\det\left[
		2T \choose T+y_k-x_j
		\right]_{j,k=1}^N.
	\end{align*}
	Such a model is sometimes referred to as vicious random walkers which means that the walkers can not be in the same position at each discrete step. It was shown in \cite{forrester02} that this binomial determinant is related to symmetric Hahn polynomials. However, the following proposition demonstrates that such a binomial determinant could also be related to a Toeplitz determinant and orthogonal polynomials on the unit circle (OPUC).
	\begin{proposition}(\cite{gharakhloo24} with a correction)
		The binomial coefficient could be written by a residue formula
		\begin{align*}
			{2T \choose T+y_k-x_j}
			=\int_{\mathbb{T}} \omega(z)z^{-(y_k-x_j)}\frac{dz}{2\pi i z}:=m_{y_k-x_j}, \quad\omega(z)=(1+z)^{2T}z^{-T},
		\end{align*}
		where $\mathbb{T}$ denotes a unit circle. Moreover, if $x_j=y_j=j$ for $j=1,2,\cdots,N$, then the binomial determinant could be expressed as a Toeplitz determinant $\det(m_{k-j})_{j,k=1}^N$.
	\end{proposition}
	Similar to the Proposition \ref{prop2.6}, if we define
	\begin{align*}
		P_n(x)=\left|\begin{array}{cccc}
			m_0&m_1&\cdots&m_{n}\\
			m_{-1}&m_0&\cdots&m_{n-1}\\
			\vdots&\vdots&&\vdots\\
			m_{-n+1}&m_{-n+2}&\cdots&m_{1}\\
			1&z&\cdots&z^n
		\end{array}
		\right|,
	\end{align*}
	then one could obtain the following orthogonal relation
	\begin{align*}
		\int_{\mathbb{T}} \omega(z)P_n(z)\bar{P}_m(\bar{z})\frac{dz}{2\pi iz}=H_n\delta_{n,m},
	\end{align*}
	where $\bar{z}$ is the complex conjugate of $z$ and 
	$H_n=\det(m_{k-j})_{k,j=0}^{n-1}\det(m_{k-j})_{k,j=0}^{n}.$ Regarding with general $\{x_j,y_j\}$ and slat Toeplitz determinants, one could refer to \cite{gharakhloo24,gharakhloo23} for more examples.
	
	\section{Stembridge's results on Pfaffians and applications into skew-orthogonal polynomials}\label{sec3}
	In \cite{stembridge90}, Stembridge generalized the result of Lindstr\"om-Gessel-Viennot, by considering non-intersecting paths from a set of vertices to an ordered subset $I\subset V$, whose number of vertices is indefinite. In this section, let's assume that $\mathbf{u}=\left(u^{(1)}, u^{(2)}, \ldots, u^{(r)}\right)$ is a finite sequence of vertices in $V$.
	
	\begin{definition}[$\mathscr{P}({u};I)$, $\mathscr{P}(\mathbf{u};I)$, $\mathscr{P}_0(\mathbf{u};I)$]
		Let $\mathscr{P}({u};I)$ be the set of all paths from $u$ to any $v\in I$, $\mathscr{P}(\mathbf{u}; I)$ be the r-tuple set of paths $(P_1,\cdots,P_r)$ where $P_i\in\mathscr{P}(u^{(i)},I)$, and  
		$\mathscr{P}_0(\mathbf{u};I)$ be the non-intersecting paths in 	$\mathscr{P}(\mathbf{u};I)$.
	\end{definition}
	\begin{definition}[$Q_I(\mathbf{u})$]
		$Q_I(\mathbf{u})$ refers to the generating function of all non-intersecting paths from $\mathbf{u}$ to $I$ in the graph $D$. Literally, we can write $Q_I(\mathbf{u})=Q_I\left(u^{(1)}, u^{(2)}, \ldots, u^{(r)}\right)=\gf\left[\mathscr{P}_0(\mathbf{u} ; I)\right]$.
	\end{definition}
	Especially, if $\mau=\{u\}$, which is a single vertex, then 
	\begin{align}\label{qi}
		Q_I(u)=\sum_{v\in I}h(u,v).
	\end{align}
	If $\mau=(u^{(1)},u^{(2)})$, then according to LGV theorem, we have
	\begin{align}\label{qij}
		\begin{aligned}
			Q_I(u^{(1)},u^{(2)})&=\sum_{v^{(1)}<v^{(2)}\in I}\det\left(\begin{array}{cc}
				h(u^{(1)},v^{(1)})&h(u^{(1)},v^{(2)})\\
				h(u^{(2)},v^{(1)})&h(u^{(2)},v^{(2)})
			\end{array}
			\right)\\
			&=\sum_{v^{(1)},v^{(2)}\in I}h(u^{(1)},v^{(1)})h(u^{(2)},v^{(2)})\sgn(v^{(2)}-v^{(1)}),
		\end{aligned}
	\end{align}
	where $\sgn$ is a sign function defined by
	\begin{align*}
		\sgn(a)=\left\{\begin{array}{ll}
			1,&\text{if $a>0$},\\
			0,&\text{if $a=0$},\\
			-1,&\text{if $a<0$}.
		\end{array}
		\right.
	\end{align*}
	\begin{remark}
		The generating function for 2-points \eqref{qij} naturally induces a skew-symmetric inner product. If we define 
		\begin{align}\label{sin}
			\begin{aligned}
				\langle\cdot,\cdot\rangle:\,&\mathbb{R}[x]\times\mathbb{R}[x]\to\mathbb{R}\\
				&\langle f_1(x),f_2(x)\rangle\mapsto \int_{\mathbb{R}\times\mathbb{R}}f_1(x)f_2(y)\sgn(y-x)w(x)w(y)dxdy,
			\end{aligned}
		\end{align}
		then this inner product is skew symmetric, i.e. we have $\langle f_1(x),f_2(x)\rangle=-\langle f_1(x),f_2(x)\rangle$. Moreover, if we take
		\begin{align*}
			f_1(x)=h(u^{(1)},x),\quad f_2(x)=h(u^{(2)},x),\quad w(x)=\sum\delta_{x\in I},
		\end{align*}
		then we see that $Q_I(u^{(1)},u^{(2)})=\langle f_1(x),f_2(x)\rangle.$
	\end{remark}
	It was found by Stembridge that $\gf[\mathscr{P}_0(\mathbf{u};I)]$ could be written as a Pfaffian, whose elements are given by \eqref{qi} and \eqref{qij}. To this end, let's introduce the concept of 1-factor, and the graphic definition of Pfaffians. 
	
	\begin{definition}[1-Factor]
		Let $\mau=\left(u^{(1)}, u^{(2)}, \ldots, u^{(2n)}\right)$. For each perfect matching for these vertices, the set of all pairings is called a 1-factor. Denote the set of all 1-factors for $\mathbf{u}$ as $\mathscr{F}(\mathbf{u})$. Specifically, if the set consists of the first $2n$ consecutive natural numbers, then the set of 1-factors can be simply denoted as $\mathscr{F}_{2n}$.
	\end{definition}
	
	For example, for a set $\mau=\left({u^{(1)}},{u^{(2)}},{u^{(3)}},{u^{(4)}}\right)$, the perfect matching $\left\{ \left\{ {u^{(1)}, u^{(2)}} \right\}, \left\{ {u^{(3)}, u^{(4)}} \right\} \right\}$ is a 1-factor in  $\mathscr{F}(\mathbf{u})$.

	\begin{definition}[Crossing Number]
		For any 1-factor $\pi \in \mathscr{F}(\mathbf{u})$ for the set $\left(u^{(1)}, u^{(2)}, \ldots, u^{(2n)}\right)$, if we align vertices in a straight line with order $u^{(1)}<u^{(2)}<\cdots<u^{(2n)}$, and connect the paired ones with curves above the line, then the number of intersections of these curves is defined as the crossing number, denoted by $\operatorname{cs}(\pi)$. The sign of $\pi$ is defined as $\operatorname{sgn}(\pi)=(-1)^{\operatorname{cs}(\pi)}$.
		
		
		
		
		
	\end{definition}
	
	\begin{definition}[Definition of Pfaffians]
		If $A=\left[a_{i,j}\right]_{1 \leqslant i,j \leqslant 2n}$ is a skew symmetric matrix of order $2n$ such that $a_{i,j}=-a_{j,i}$, then we can define the Pfaffian of $A$ by
		\begin{align}\label{expansionformula}
			\pf(A)=\sum_{\pi \in \mathscr{F}_{2n}} \operatorname{sgn}(\pi) \prod_{(i, j) \in \pi} a_{i j}.
		\end{align}
	\end{definition}

	\subsection{The generating function for non-intersecting paths from $\mau\to I$}
	The following result was given by Stembridge in \cite[Thm. 3.1]{stembridge90}.
	\begin{theorem}\label{stem1}
		
		Let $\mathbf{u}=\left(u^{(1)}, u^{(2)}, \ldots, u^{(2n)}\right)$ be $2n$-tuple of vertices in an acyclic directed graph $D$. If $I \subset V$ is a totally ordered subset of vertices such that $\mathbf{u}$ is D-compatible with $I$, then
		$$
		Q_I\left(u^{(1)}, u^{(2)}, \ldots, u^{(2n)}\right)=\operatorname{Pf}\left[Q_I (u^{(i)}, u^{(j)})\right]_{1 \leqslant i,j \leqslant 2n},
		$$
		where $Q_I(u^{(i)},u^{(j)})$ is given by \eqref{qij}. Moreover, if $\mau=(u^{(1)},u^{(2)},\cdots,u^{(2n+1)})$, then the generating function of $\mau\to I$ is given by
		\begin{align*}
			Q_I\left(u^{(1)},u^{(2)},\cdots,u^{(2n+1)}\right)=\operatorname{Pf}\left[
			\begin{array}{cc}
				Q_I(u^{(i)},u^{(j)}) &Q_I(u^{(i)})\\
				-Q_I(u^{(j)})&0
			\end{array}
			\right]_{1\leq i,j\leq 2n+1},
		\end{align*}
		where $Q_I(u^{(i)})$ is given by \eqref{qi} and $Q_I(u^{(i)},u^{(j)})$ is given by \eqref{qij}.
	\end{theorem}
	\begin{remark}
		It should be remarked that if we consider the generating function for non-intersecting paths from $J$ to $\mav$, then we have the expression
		\begin{align*}
			\gf[\Mp_0(J;\mav)]=\tilde{Q}_J(v^{(r)},\cdots,v^{(1)}),
		\end{align*}
		where 
		\begin{align*}
			\tilde{Q}_J(v)=\sum_{u\in J}h(u,v),\quad \tilde{Q}_J(v^{(2)},v^{(1)})=\sum_{u^{(1)},u^{(2)}\in J}h(u^{(1)},v^{(1)})h(u^{(2)},v^{(2)})\sgn(u^{(2)}-u^{(1)}),
		\end{align*}
		and $\gf[\Mp_0(J;\mav)]$ could be computed by above 1-vertex and 2-vertex formulas for any $r\in\mathbb{N}_+$.
	\end{remark}

		\subsection{The generating function of non-intersecting paths from $\mau\to\mav\oplus I$}
		In this part, let's assume $\mau=(u^{(1)},u^{(2)},\cdots,u^{(r)})$ and $\mav=(v^{(1)},v^{(2)},\cdots,v^{(s)})$ with $s\leq r$.
		\begin{definition}[$\mathbf{v} \oplus I$]\label{voi}
			$\mathbf{v} \oplus I$ is defined as a union of vertices $\mathbf{v}$ and $I$, where $\mav$ and $I$ are disjoint, and all vertices are arranged such that every $v^{(i)}$ in $\mathbf{v}$ precedes vertices in $I$. 
		\end{definition}
		\begin{definition}[$\mathscr{P}_0(\mathbf{u}; \mathbf{v} \oplus I)$]
			$\mathscr{P}_0(\mathbf{u}; \mathbf{v} \oplus I)$ is defined as the set of all non-intersecting paths  $\left(P_1, \ldots, P_r\right)$, where for $1\leq i\leq s$, $P_i\in\Mp(u^{(i)};v^{(i)})$; while for $s+1\leq i\leq r$, $P_i\in\Mp(u^{(i)};I)$.
		\end{definition}
		
		\begin{theorem}[\cite{stembridge90} with a supplement statement]\label{stem2}
			
			Let $\mathbf{u}=\left(u^{(1)}, u^{(2)}, \ldots, u^{(r)}\right)$ and $\mathbf{v}=\left(v^{(1)}, \ldots, v^{(s)}\right)$ be sequences of vertices in the acyclic directed graph $D$, and suppose $I$ is a totally ordered subset of $V$ such that $\mathbf{u}$ and $\mathbf{v} \oplus I$ are $D$-compatible. If $r+s$ is even, then
			\begin{align*}
				\gf \left[ \mathscr{P}_0({\bf{u}};{\bf{v}} \oplus I) \right] = {\mathop{\rm Pf}\nolimits} \left[ \begin{array}{cc}
					A&B\\
					-B^\top&\mathbf{0}
				\end{array}
				\right],
			\end{align*}
			where 
			\begin{align}\label{matrixab}
				A=\left(
				Q_I(u^{(i)},u^{(j)})
				\right)_{1\leq i,j\leq r},\quad B=\left(
				h(u^{(i)},v^{(s+1-j)})
				\right)_{1\leq i\leq r,\,1\leq j\leq s}.
			\end{align}
			Moreover, if $r+s$ is odd, then
			\begin{align*}
				\gf\left[\mathscr{P}_0(\mathbf{u} ; \mathbf{v} \oplus I)\right]=\operatorname{Pf}\left[\begin{array}{cc}
					\hat{A}&\hat{B}\\
					-\hat{B}^\top &0
				\end{array}\right],
			\end{align*}
			where 
			\begin{align*}
				\hat{A}=\left(\begin{array}{cc}
					A&\vec{\alpha}\\
					-\vec{\alpha}^\top&0
				\end{array}
				\right),\hat B = \left( {\begin{array}{*{20}{c}}
						B\\
						0
				\end{array}} \right),\quad \vec{\alpha}=\left(
				Q_I(u^{(i)})
				\right)_{1\leq i\leq r},
			\end{align*}		
			and matrices $A$ and $B$ are given by \eqref{matrixab}.
			
		\end{theorem}

				\begin{proof}
					Here we give a proof when $r+s$ is odd, which is omitted in Stembridge's paper. We assume that there is an auxiliary vertex $u^{(r+1)}$ added in the graph $D$. In order to satisfy $D$-compatibility, we might as well assume that this vertex follows all vertices in $\mathbf{u}$ and $I$. Furthermore, we require that weights of $u^{(r+1)}$ to itself is $1$, and to other vertices are zero. Let's denote $\mathbf{u}^*=\left(u^{(1)}, \ldots, u^{(r)}, u^{(r+1)}\right)$, $I^*=I \cup\left\{u^{(r+1)}\right\}$, then one could show that 
					\begin{align*}
						\mathrm{GF}\left[\mathscr{P}_0(\mathbf{u} ; \mathbf{v} \oplus I)\right]=\operatorname{GF}\left[\mathscr{P}_0\left(\mathbf{u}^* ; \mathbf{v} \oplus I^*\right)\right],
					\end{align*}
					which gives the result.					
				\end{proof}
				
				It should be remarked that the Theorem \ref{stem2} comprises previous results in Theorem \ref{lgv} and Theorem \ref{stem1}. If $\mav$ is an empty set, then it degenerates to Theorem \ref{stem1}, while if $I$ is an empty set and $r=s$, then it degenerates to the LGV lemma. Therefore, we want to generalize Stembridge's result to consider combinatoric explanations for full block Pfaffians. Before that, a comparison to Ishikawa-Wakayama's result is made to show the strength of Stembridge's method.
				
				\subsection{A comparison to Ishikawa-Wakayama's result}		
				In \cite[Theorem 4.3]{ishiwaka05}, Ishikawa and Wakayama obtained similar results to Theorem \ref{stem1}, which reads $$
				\sum_{S \in{I \choose m}} \operatorname{Pf}\left(A_S^S\right) \sum_{\pi \in S_m} \operatorname{sgn}(\pi) \operatorname{GF}\left[\mathscr{P}_0\left(u^\pi, I\right)\right]=\operatorname{Pf}(H A H^\top).
				$$
				In the above formula, $\mathbf{u}=\left(u^{(1)} , \ldots, u^{(m)}\right)$ is a set of points consisting of an even number of elements, $I=\left\{V_1<\cdots<V_N\right\}$ is a finite set of even-number vertices, and ${I \choose m}$ represents the set of all subsets of $I$ containing exactly $m$ elements. Moreover, $A=\left(a_{V_i V_j}\right)_{1 \leq i<j \leq N}$ is a skew-symmetric matrix and  $A_S^S$ is the submatrix of $A$ obtained
				by picking up the rows and  columns indexed by S, and $H=\left(h\left(u^{(i)}, V_j\right)\right)_{\substack{1 \leq i \leq m ,1 \leq j \leq N}}$. 
				
				To relax the condition of $D$-compatibility,  Ishikawa and Wakayama imposed an additional condition that the set of vertices $I$ should be finite and the number of vertices should be even. When $\mathbf{u}$ and $I$ are $D$-compatible, we can get the same result stated as in Theorem \ref{stem1}.
				Moreover, Ishikawa and Wakayama generalized Theorem \ref{stem2} to  \cite[Theorem 4.4]{ishiwaka05}, where the $D$-compatible condition is relaxed. For both $(m+n)$ and $N$ being even, $\mathbf{v}=\left(v^{(1)}, v^{(2)}, \ldots, v^{(n)}\right)$, we have $$
				\sum_{s \in{I \choose m-n}} \operatorname{Pf}\left(A_S^S\right) \sum_{\pi \in S_m} \operatorname{sgn}(\pi) \mathrm{GF}\left[\mathscr{P}_0\left(u^\pi, S^0 \oplus I\right)\right]=\operatorname{Pf}\left(\begin{array}{cc}
					H A H^{\top} & H J_N \\
					-J_N H^{\top} & \mathbf{0}
				\end{array}\right),
				$$ where $$
				J_N=\left(\begin{array}{cccc}
					0 & \ldots & 0 & 1 \\
					0 & \ldots & 1 & 0 \\
					\vdots & . & \vdots & \vdots \\
					1 & \ldots & 0 & 0
				\end{array}\right).
				$$

				Although the results of Stembridge and Ishikawa-Wakayama are both non-intersecting path explanations for Pfaffians, Stembridge derived the general generating function from $\mathbf{u}$ to $I$ by using the 1-vertex generating function \eqref{qi} and 2-vertex generating function \eqref{qij}, while Ishikawa-Wakayama derived the general generating function by the weight $h\left(u^{(i)}, V_j\right)$ between points in $\mathbf{u}$ and $I$, which contains more imformation between $\mathbf{u}$ and $I$. This is the reason why $D$-compatibility could be relaxed. Furthermore, Ishikawa-Wakayama's proof avoided internal pairings between points in $\mathbf{u}$, which made it easier to find a map to offset all intersecting paths. Since we don't want to make any constraints on the terminal interval $I$, we still adopt Stembridge's method in the following sections.
				
				\subsection{Application of Stembridge's theorem}	
				There have been numerous application of Stembridge's graphic explanation for Pfaffians. 
				In \cite{stembridge90,okada89}, it was used to count skew Young tableaux and the enumeration of plane partitions, and later it was used to demonstrate a minor summation formula for Pfaffian in \cite{ishiwaka05}. In recent years, a fusion of combinatorial technique and statistical physics leads to more applications of Stembridge's result. For example, Theorem \ref{stem1} was used in a vicious random walker problem without return \cite{forrester02}, and Theorem \ref{stem2} was used in a free-boundary lozenge tiling problem \cite[Section 4]{ciucu11} where some triangular holes in the lattices were imposed. In this part, we demonstrate how to introduce skew-orthogonal polynomials by considering non-intersecting Brownian motions and applications of Stembridge's results.
				
				Let's first consider $N$ independent simple random walks $X(t)=(X_1(t),\cdots, X_N(t))$ starting from $\vec{\alpha}=(\alpha_1,\cdots,\alpha_N)$ as in Section \ref{sec2}. If we consider a continuous-time non-intersecting Brownian motion in the configuration space $\mathbb{R}\times\mathbb{R}_{\geq 0}$, then it is known that at an arbitrary time $T$, the distribution of arrival points at $\vec{x}=(x_1,\cdots,x_N)\in D_N$ could be formulated by 
				\begin{align}\label{pdf}
					\rho_{T}(x_1,\cdots,x_N)=\frac{1}{Z_N}
					\det\left(p(\alpha_j,x_k;T)\right)_{j,k=1}^{N}
				\end{align}
				with a normalization factor $Z_N$.
				If $N=2n$ is even, then $Z_{2n}$ could be written by
				\begin{align*}
					Z_{2n}=\int_{x_1<\cdots<x_{2n}}\det\left(
					p(\alpha_j,x_k;T)
					\right)_{j,k=1}^Nd\vec{x}=\pf\left[
					Q_{\mathbb{R}} (\alpha_j,\alpha_k)
					\right]_{j,k=1}^{2n},
				\end{align*}
				where 
				\begin{align*}
					Q_\mathbb{R}(\alpha_j,\alpha_k)=\int_{\mathbb{R}^2}p(\alpha_j,x;T)\sgn(y-x)p(\alpha_k,y;T)dxdy.
				\end{align*}
				Moreover, if $N=2n+1$ is odd, then
				\begin{align*}
					Z_{2n+1}=\int_{x_1<\cdots<x_{2n+1}}\det\left(
					p(\alpha_j,x_k;T)
					\right)_{j,k=1}^{2N+1}d\vec{x}=\pf\left[\begin{array}{cc}
						Q_{\mathbb{R}}(\alpha_j,\alpha_k)&Q_{\mathbb{R}}(\alpha_j)\\
						Q_{\mathbb{R}}(\alpha_k)&0
					\end{array}
					\right]_{j,k=1}^{2n+1}
				\end{align*}
				with $Q_{\mathbb{R}}(\alpha_j)=\int_{\mathbb{R}}p(\alpha_j,x;T)dx$.
				It should be noted that the element in this Pfaffian coincide with the formula in the skew-symmetric inner product \eqref{sin}. Moreover, we have the following proposition.
				\begin{proposition}
					Let's define monic functions
					\begin{align*}
						P_{2n}(x)&=\frac{1}{Z_{2n}}\pf\left[\begin{array}{ccc}
							Q_{2n}&v_{2n+1}&\Psi_{2n}(x)\\
							-v_{2n+1}^\top&0&\psi_{2n+1}(x)\\
							-\Psi^\top_{2n}(x)&-\psi_{2n+1}(x)&0
						\end{array}
						\right],\\ P_{2n+1}(x)&=\frac{1}{Z_{2n}}\pf\left[\begin{array}{ccc}
							Q_{2n}&\Psi_{2n}(x)&v_{2n+2}\\
							-\Psi_{2n}^\top(x)&0&\psi_{2n+2}(x)\\
							-v_{2n+2}^\top&-\psi_{2n+2}(x)&0
						\end{array}
						\right],
					\end{align*}
					where $Q_{2n}=(Q_{\mathbb{R}}(\alpha_j,\alpha_k))_{j,k=1}^{2n}$, $\Psi_{2n}=(p(\alpha_1,x;T),\cdots,p(\alpha_{2n},x;T))^\top$, and $v_{2n+1}$ and $v_{2n+2}$ are two column vectors admitting the forms  $v_i=(Q_{\mathbb{R}}(\alpha_1,\alpha_{i}),\cdots,Q_{\mathbb{R}}(\alpha_{2n},\alpha_{i}))^\top$ for $i=2n+1,\,2n+2$. Moreover, $\{P_{k}(x)\}_{k\in\mathbb{N}}$ satisfy  skew orthogonality relations
					\begin{align*}
						&\int_{\mathbb{R}^2}P_{2n}(x)\sgn(y-x)P_{2m}(y)dxdy=\int_{\mathbb{R}^2}P_{2n+1}(x)\sgn(y-x)P_{2m+1}(y)dxdy=0,\\
						&\int_{\mathbb{R}^2}P_{2n}(x)\sgn(y-x)P_{2m+1}(y)dxdy=\frac{Z_{2n+2}}{Z_{2n}}\delta_{n,m}.
					\end{align*}
				\end{proposition}
				It should be remarked that there are only $Z_{2n}$ involved under the frame of skew-orthogonal polynomials. To have a clear understanding of $Z_{2n+1}$, partial-skew-orthogonal polynomials should be introduced and we have the following proposition.
				
				\begin{proposition}
					If we define monic functions
					\begin{align*}
						R_{2n}(x)&=\frac{1}{Z_{2n}}\pf\left[\begin{array}{cc}
							Q_{2n+1}&\Psi_{2n+1}(x)\\
							-\psi^\top_{2n+1}(x)&0
						\end{array}
						\right],\\R_{2n+1}(x)&=\frac{1}{Z_{2n+1}}\pf\left[\begin{array}{ccc}
							0&w_{2n+2}&0\\
							-w_{2n+2}^\top&Q_{2n+2}&\Psi_{2n+2}(x)\\
							0&-\Psi_{2n+2}(x)&0
						\end{array}
						\right],
					\end{align*}
					where $Q_{2n+2}=(Q_{\mathbb{R}}(\alpha_j,\alpha_k))_{j,k=1}^{2n+2}$, $\Psi_{2n+2}(x)=(p(\alpha_1,x;T),\cdots,p(\alpha_{2n+2},x;T))^\top$ and $w_{2n+2}=(Q_{\mathbb{R}}(\alpha_1),\cdots,Q_{\mathbb{R}}(\alpha_{2n+2}))$. Moreover, $\{R_k(x)\}_{k\in\mathbb{N}}$ are determined by the following relations
					\begin{align*}
						&\int_{\mathbb{R}^2} R_{2n}(x)\sgn(y-x)p(\alpha_m,x;T)dxdy=\frac{Z_{2n+2}}{Z_{2n}}\delta_{2n+1,m},\qquad\qquad \text{for $m\leq 2n+1$},\\
						&\int_{\mathbb{R}^2}R_{2n+1}(x)\sgn(y-x)p(\alpha_m,x;T)dxdy=-\frac{Z_{2n+2}}{Z_{2n+1}}Q_\mathbb{R}(\alpha_m),\qquad \text{for $m\leq 2n+1$}.
					\end{align*}
				\end{proposition}
				
				\section{Non-intersecting path explanations for block Pfaffians and applications}\label{sec4}
				
				In this section, we generalize Stembridge's theorem and obtain a block matrix representation for generating function of non-intersecting paths. We assume that these paths should start from and end at several different sets of vertices. 
				
				\subsection{Pfaffian representations for $J\oplus \mathbf{u}\to \mathbf{v}\oplus I$}
				To give a Pfaffian representation for paths $J\oplus \mathbf{u}\to\mathbf{v}\oplus I$, we need to introduce the separation of a graph.		
				
				\begin{definition}[$D$-separated] We call vertex sets $J\oplus\mathbf{u}$ and $\mathbf{v}\oplus I$ to be $D$-separated if paths from $J$ to $\mathbf{v}$ and those from $\mathbf{u}$ to $I$ are non-intersecting.  
				\end{definition}	
				
				\begin{definition}[$\mathscr{P}_0(J \oplus \mathbf{u} ; \mathbf{v} \oplus I)$]
					If a graph $D$ is $D$-separated, then $\mathscr{P}_0(J \oplus \mathbf{u} ; \mathbf{v} \oplus I)$ represents all non-intersecting paths $J\oplus \mathbf{u}\to\mathbf{v}\oplus I$ without any path from $J\to I$. Moreover, paths should be consecutive; that is, if there is a path from $u^{(1)}\to v^{(k)}$, then we must have $u^{(2)}\to v^{(k+1)}$ and so on, until that $v^{(s)}$ is matched over. 
				\end{definition}
				
				Again, we assume that $\mathbf{u}=(u^{(1)},\cdots,u^{(r)})$ and $\mathbf{v}=(v^{(1)},\cdots,v^{(s)})$. If $r+s$ is even, then we use ${}_e^{e}\mathscr{P}_0(J\oplus \mathbf{u};\mathbf{v}\oplus I)$ (respectively ${}_o^{o}\mathscr{P}_0(J\oplus \mathbf{u};\mathbf{v}\oplus I)$ ) to denote the number of  paths from $J\to\mathbf{v}$ and $\mathbf{u}\to I$ are both even (respectively both odd). Here we don't have the case where the number of paths from $J\to \mathbf{v}$ is even, and that from $\mathbf{u} \to I$ is odd. Otherwise there will be some paths from $J\to I$, which is contradicted with $D$-separated . 
				On the other hand, if $r+s$ is odd, then we use  ${}_e^{o}\mathscr{P}_0(J\oplus \mathbf{u};\mathbf{v}\oplus I)$ (respectively ${}_o^{e}\mathscr{P}_0(J\oplus \mathbf{u};\mathbf{v}\oplus I)$) to represent that there are even (respectively odd) paths from $\mathbf{u}\to I$ and odd (respectively even) paths from $J\to\mathbf{v}$. Let's denote weight matrices
				\begin{align*}
					&Q=\left(
					Q_I(u^{(i)},u^{(j)})
					\right)_{1\leq i,j\leq r},\quad \tilde{Q}=\left(
					\tilde{Q}_J(v^{(s+1-i)},v^{(s+1-j)})
					\right)_{1\leq i,j\leq s},\\
					& H=\left(
					h(u^{(i)},v^{(s+1-j)})
					\right)_{1\leq i\leq r,1\leq j\leq s},
				\end{align*}
				and column vectors $Q_I=\left(
				Q_I(u^{(i)})
				\right)_{1\leq i\leq r}$
				and $\tilde{Q}_J=(\tilde{Q}_J(v^{(j)}))_{1\leq j\leq s}$.
				
				\begin{theorem}\label{thm4.3}
					Let $\mathbf{u}=\left(u^{(1)}, \ldots, u^{(r)}\right)$ and $\mathbf{v}=\left(v^{(1)}, \ldots, v^{(s)}\right)$ be sequences of vertices in an acyclic directed graph $D$. Assume that $I$ and $J$ are totally ordered subsets of $V$, such that $J \oplus \mathbf{u}$ and $\mathbf{v} \oplus I$ are $D$-compatible and $D$-separated. If $r+s$ is even, then we have
					\begin{align}\label{pf1}
						\gf[{}_e^{e}\mathscr{P}_0(J\oplus \mathbf{u};\mathbf{v}\oplus I)]=\pf\left[
						\begin{array}{cc}
							Q&H\\
							-H^\top&\tilde{Q}
						\end{array}
						\right],
					\end{align}	
					and 
					\begin{align}\label{pf4.2}
						\gf[{}_o^{o}\mathscr{P}_0(J\oplus \mathbf{u};\mathbf{v}\oplus I)]=\operatorname{Pf}\left[\begin{array}{cccc}
							Q & Q_I & H & 0 \\
							-Q_I^{\top} & 0 & 0 & 0 \\
							-H^{\top} & 0 & \tilde{Q} & \tilde{Q}_J \\
							0 & 0 & -\tilde{Q}_J^{\top} & 0
						\end{array}\right].
					\end{align}	
					On the other hand, if $r+s$ is odd, then
					\begin{align}\label{pf3}
						\gf[{}_o^{e}\mathscr{P}_0(J\oplus \mathbf{u};\mathbf{v}\oplus I)]=\operatorname{Pf}\left[\begin{array}{ccc}
							Q & H & 0 \\
							-H^{\top} & \tilde{Q} & \tilde{Q}_J \\
							0 & -\tilde{Q}_J^{\top} & 0
						\end{array}\right],
					\end{align}	
					and
					\begin{align}\label{pf4}
						\gf[{}_e^{o}\mathscr{P}_0(J\oplus \mathbf{u};\mathbf{v}\oplus I)]=\operatorname{Pf}\left[\begin{array}{ccc}
							Q & Q_I & H \\
							-Q_I^{\top} & \tilde{Q} & 0 \\
							-H^{\top} & 0 & 0
						\end{array}\right].
					\end{align}	
				\end{theorem}		
				\begin{proof}
					Let's prove the equation \eqref{pf1}, and \eqref{pf4.2}-\eqref{pf4} could be similarly verified. This proof is based on \emph{the path switching involution for non-intersecting  paths}, and is divided into three steps.
					
					\noindent
					(1) {\bf Expanding the Pfaffian as a summation of configurations. } 
					By expanding \eqref{pf1}, we have
					\begin{align*}
						\pf\left[\begin{array}{cc}
							Q&H\\-H^\top&\tilde{Q}
						\end{array}
						\right]&=\sum_{\pi} \text{sgn}(\pi)\times\\
						&\quad\prod_{(u^{(i)},u^{(j)})\in\pi} Q_I(u^{(i)},u^{(j)})\prod_{(u^{(i)},v^{(j)})\in\pi}h(u^{(i)},v^{(j)})\prod_{(v^{(i)},v^{(j)})\in\pi} \tilde{Q}_J(v^{(j)},v^{(i)}),
					\end{align*}
					where $\pi$ belongs to the set $\mathscr{F}(u^{(1)},\cdots,u^{(r)},v^{(s)},\cdots,v^{(1)})$. We recognize this formula as a product of path weights by rewriting it as 
					\begin{align}\label{sum}
						\pf\left[\begin{array}{cc}
							Q&H\\-H^\top&\tilde{Q}
						\end{array}
						\right]=\sum_{(\pi,P^{(1)},\cdots, P^{(r)},Q^{(1)},\cdots,Q^{(s)})}\text{sgn}(\pi) \omega(P^{(1)},\cdots,P^{(r)},Q^{(1)},\cdots,Q^{(s)}),
					\end{align}
					where ${(\pi,P^{(1)},\cdots, P^{(r)},Q^{(1)},\cdots,Q^{(s)})}$ is a configuration such that

					\begin{itemize}
						\item If $(u^{(i)},v^{(j)})\in \pi$, then $P^{(i)}=Q^{(j)}\in\mathscr{P}(u^{(i)},v^{(j)})$;
						\item If $(u^{(i)},u^{(j)})\in\pi$, then $P^{(i)}\in \mathscr{P}(u^{(i)},I)$, $P^{(j)}\in\mathscr{P}(u^{(j)},I)$, while $P^{(i)}$ and $P^{(j)}$ don't intersect;
						\item If $(v^{(i)},v^{(j)})\in\pi$, then $Q^{(i)}\in \mathscr{P}(J, v^{(i)})$, $Q^{(j)}\in\mathscr{P}(J,v^{(i)})$, while $Q^{(i)}$ and $Q^{(j)}$ don't intersect.
						
					\end{itemize}
					We call these three conditions as the configuration conditions.

					\noindent 
					(2) {\bf Finding a sign-reversing involution to offset the intersecting paths.}
					To make the RHS in \eqref{sum} represent the summation of weights for all non-intersecting paths, we intend to show that there exists an involution which can offset the intersecting term. Since the graph is $D$-separated, we know that if there are paths intersected, then the intersecting point belongs to $P^{(i)}\cap P^{(j)}$ or $Q^{(i)}\cap Q^{(j)}$ for some $i$ and $j$. We show the first case and the latter could be similarly verified. Let's find the smallest intersecting point $w^*$\footnote{We call an intersecting point the smallest if it is the first intersecting point starting from $u^{(1)}$. If there isn't any intersecting point in $P^{(1)}$, then we find it in $P^{(2)}$, and so on.}(the biggest intersecting point $w^*$ in the latter case). Among the paths passing through $w^*$, we assume that $P^{(i)}$ and $P^{(j)}$ are the two paths which admit two smallest indices.  Thus we could construct an involution
					\begin{align*}
						\phi: (\pi,P^{(1)},\cdots,P^{(r)},Q^{(1)},\cdots,Q^{(s)})\to (\bar{\pi},\bar{P}^{(1)},\cdots,\bar{P}^{(r)},\bar{Q}^{(1)},\cdots,\bar{Q}^{(s)}),
					\end{align*}
					where we define paths $\bar{P}^{(i)}$ and $\bar{P}^{(j)}$ as
					\begin{align*}
						\bar{P}^{(i)}=P^{(i)}(\to w^*)P^{(j)}(w^*\to),\quad \bar{P}^{(j)}=P^{(j)}(\to w^*)P^{(i)}(w^*\to),
					\end{align*}
					and the other paths remain the same. At the same time, $\bar{\pi}$ is obtained by interchanging $u^{(i)}$ and $u^{(j)}$ in $\pi$. Next, we show that $(\bar{\pi},\bar{P}^{(1)},\cdots,\bar{P}^{(r)},\bar{Q}^{(1)},\cdots,\bar{Q}^{(s)})$ is still a configuration, and $\phi$ is sign-reversing. 
					
					To prove this, we first demonstrate that $\left(\bar{\pi}, \bar{P}^{(1)}, \ldots, \bar{P}^{(r)}, \bar{Q}^{(1)}, \ldots, \bar{Q}^{(s)}\right)$ is still in the summation \eqref{sum}, which is equivalent to verify that $\left(\bar{\pi}, \bar{P}^{(1)}, \ldots, \bar{P}^{(r)}, \bar{Q}^{(1)}, \ldots, \bar{Q}^{(s)}\right)$ satisfies configuration conditions. Since the paths except $P^{(i)}$ and $P^{(j)}$ remain the same, it is necessary to consider cases in $\bar{\pi}$ which only contain $u^{(i)}$ or $u^{(j)}$.
					If $\left(u^{(i)}, u^{(k)}\right) \in \bar{\pi}$, then we have $\left(u^{(j)}, u^{(k)}\right) \in \pi$, implying that $P^{(j)} \in \mathscr{P}\left(u^{(j)} ; I\right), P^{(k)} \in \mathscr{P}\left(u^{(k)} ; I\right)$. Therefore, as $\bar{\pi}$ is obtained by interchanging $u^{(i)}$ and $u^{(j)}$ in $\pi$, it is known that $\bar{P}^{(i)} \in \mathscr{P}\left(u^{(i)}; I\right)$ and $\bar{P}^{(k)} \in \mathscr{P}\left(u^{(k)} ; I\right)$ as required. To show that $\bar{P}^{(i)}$ and $\bar{P}^{(k)}$ do not intersect, we verify it by contradiction. Assume that $\bar{P}^{(i)}$ and $\bar{P}^{(k)}$ intersect with each other at some point, then the intersection point could only be behind the smallest $w^*$. However, since $\bar{P}^{(i)}$ is the same as $P^{(j)}$ after the point $w^*$, it means that $P^{(j)}$ and $P^{(k)}$  intersect, which leads to a contradiction with $(u^{(j)},u^{(k)})\in\pi$.
					On the other hand, if $\left(u^{(i)}, v^{(k)}\right) \in \bar{\pi}$, then $\left(u^{(j)}, v^{(k)}\right) \in \pi$, and therefore $P^{(j)} \in \mathscr{P}\left(u^{(j)}, v^{(k)}\right)$. From this fact, we know that $\bar{P}^{(i)} \in \mathscr{P}\left(u^{(i)}, v^{(k)}\right)$ as required. To show the map is sign-reversing, by making use of $D$-compatibility, we know that for any $k$ satisfying $i<k<j$, $P^{(k)}$ must intersect with $P^{(i)}$ and $P^{(j)}$. Therefore, by using \cite[Lem. 2.1]{stembridge90}, we can deduce that this involution mapping is sign-reversing.
					
					\noindent
					(3) {\bf Simplification of the configuration.} After offsetting intersecting terms, only non-intersecting terms are left. Therefore, \eqref{sum} could be rewritten as the non-intersecting part of
					\begin{align}\label{simplified}
						\sum_{\text{$\ell$ is odd}}\sum_{\pi \in \Lambda_\ell}\text{sgn}(\pi)\prod_{\left(u^{(i)}, u^{(j)}\right) \in \pi_2^{\ell}} Q_I(u^{(i)},u^{(j)}) \prod_{i=\ell}^{s} h(u^{(i+1-\ell)},v^{(i)})\prod_{ \left(v^{(i)}, v^{(j)}\right) \in \pi_3^{\ell}} \tilde{Q}_J(v^{(j)},v^{(i)}),
					\end{align}
					where 
					\begin{align*}
						&\Lambda_\ell=\{
						\pi\in\mathscr{F}(u^{(1)},\cdots,u^{(r)},v^{(s)},\cdots,v^{(1)})|\pi=\pi_1^\ell\cup\pi_2^\ell\cup\pi_3^\ell
						\},\\
						&\pi_1^\ell=\cup_{i=\ell}^s (u^{(i+1-\ell)},v^{(i)}), \quad\pi_2^\ell\in\mathscr{F}(u^{(s+2-\ell)},\cdots,u^{(r)}),\quad \pi_3^\ell\in\mathscr{F}(v^{(1)},\cdots,v^{(\ell-1)}).
					\end{align*}
					\begin{figure}[H]
						\centering
						\includegraphics[width=10cm]{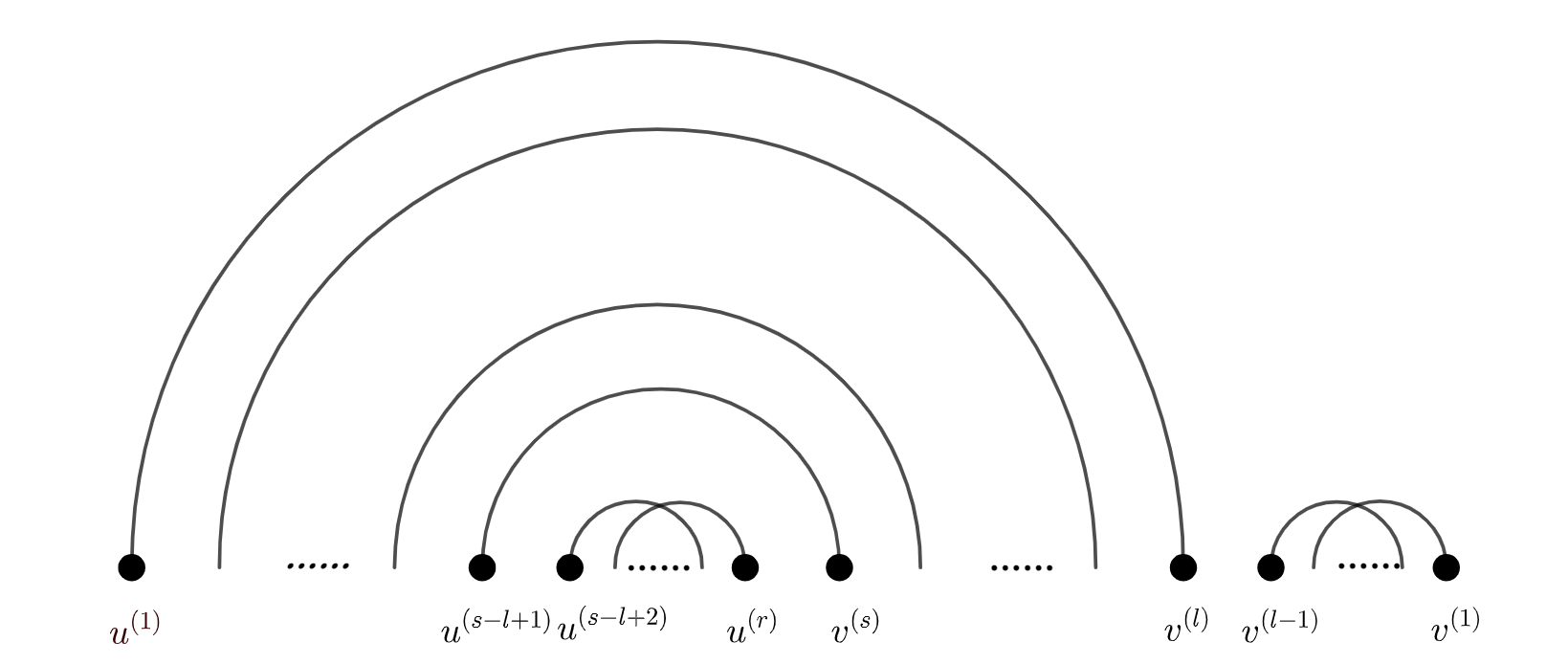}
						\caption{A decomposition of 1-factors for pairings in $J\oplus \mathbf{u}\to \mathbf{v}\oplus I$} 
					\end{figure}
					
					Moreover, by realizing
					\begin{align*}
						\sum_{\pi \in \Lambda_l} \operatorname{sgn}(\pi)=\operatorname{sgn}\left(\pi_1^{\ell}\right)\left(\sum \operatorname{sgn}\left(\pi_2^{\ell}\right)\right)\left(\sum \operatorname{sgn}\left(\pi_3^{\ell}\right)\right)=1
					\end{align*}
					we know that \eqref{simplified} represents the weight of all non-intersecting paths from $J\oplus \mathbf{u}\to \mathbf{v}\oplus I$, and there are even paths from $\mathbf{u}\to I$ and even paths from $J\to\mathbf{v}$.
				\end{proof}	
				
				Let's end with an example. If $\mau=(u^{(1)},u^{(2)})$ and $\mav=(v^{(1)},v^{(2)})$, then it is known that 
				\begin{align*}
					&\pf\left[\begin{array}{cccc}
						0&Q_I(u^{(1)},u^{(2)})&h(u^{(1)},v^{(2)})&h(u^{(1)},v^{(1)})\\
						-Q_I(u^{(1)},u^{(2)})&0&h(u^{(2)},v^{(2)})&h(u^{(2)},v^{(1)})\\
						-h(u^{(1)},v^{(2)})&-h(u^{(2)},v^{(2)})&0&\tilde{Q}_J(v^{(2)},v^{(1)})\\
						-h(u^{(1)},v^{(1)})&-h(u^{(2)},v^{(1)})&-\tilde{Q}_J(v^{(2)},v^{(1)})&0
					\end{array}
					\right]\\
					&\qquad =Q_I(u^{(1)},u^{(2)})\tilde{Q}_J(v^{(2)},v^{(1)})-h(u^{(1)},v^{(2)})h(u^{(2)},v^{(1)})+h(u^{(1)},v^{(1)})h(u^{(2)},v^{(2)}),
				\end{align*}
				where the first term represents the weight of non-intersecting paths from $\mau\to I$ and $J\to \mav$, and the rest terms represent the weight of non-intersecting paths from $u^{(1)}\to v^{(1)}$ and $u^{(2)}\to v^{(2)}$.
				
				\subsection{Applications of block Pfaffians and multiple skew-orthogonal polynomials}
				In this subsection, we show some applications of this block Pfaffians, which correspond to the multiple skew-orthogonal polynomial theory.
				
				We start with the probability density function given by \eqref{pdf}. If we consider there are only two starting points named $a_1$ and $a_2$, and there are $n_i$ paths from $\alpha_i$ for $i=1,2$, then we have a confluent form for the PDF. By noting that
				\begin{align*}
					\lim_{\alpha_1,\cdots,\alpha_{n_1}\to a_1\atop
						\alpha_{n_1+1},\cdots,\alpha_N\to a_2}\det(p(\alpha_i,x_j;T))_{i,j=1}^{N}\sim\det\left(\begin{array}{c}
						x_j^{k_1-1}p(a_1,x_j;T)\\
						x_j^{k_2-1}p(a_2,x_j;T)
					\end{array}\right)_{j=1,\cdots,N,k_1=1,\cdots,n_1,k_2=1,\cdots,n_2},
				\end{align*}	
				where $N=n_1+n_2$,				
				we know that the normalization factor in this case could be written as
				\begin{align}\label{zn1n2}
					\begin{aligned}
						Z_{n_1,n_2}&=\int_{x_1<\cdots<x_N} \det\left(\begin{array}{c}
							x_j^{k_1-1}p(a_1,x_j;T)\\
							x_j^{k_2-1}p(a_2,x_j;T)
						\end{array}\right)_{j=1,\cdots,N,k_1=1,\cdots,n_1,k_2=1,\cdots,n_2} d\vec{x}\\
						&=\pf\left[\begin{array}{cc}
							Q_{\mathbb{R}^2}(u^{(i)},u^{(j)})_{i,j=1,\cdots,n_1}&H(u^{(i)},v^{(j)})_{i=1,\cdots,n_1,\atop
								j=1,\cdots,n_2}\\
							-H(v^{(i)},u^{(j)})_{i=1,\cdots,n_2\atop j=1,\cdots,n_1}&\tilde{Q}_{\mathbb{R}^2}(v^{(i)},v^{(j)})_{i,j=1,\cdots,n_2}
						\end{array}
						\right]
					\end{aligned}
				\end{align}				
				if $n_1+n_2$ is even. Here we note that the partition function $Z_{n_1,n_2}$ is in the form of equation \eqref{pf1}. Moreover, in the above notations, we have
				\begin{align*}
					Q_{\mathbb{R}^2}(u^{(i)},u^{(j)})&=\int_{\mathbb{R}^2} x^{i-1}p(a_1,x;T)\sgn(y-x)y^{j-1} p(a_1,y;T)dxdy,\\
					\tilde{Q}_{\mathbb{R}^2}(v^{(i)},v^{(j)})&=\int_{\mathbb{R}^2}x^{i-1}p(a_2,x;T)\sgn(y-x)y^{j-1}p(a_2,y;T)dxdy,\\
					H(u^{(i)},v^{(j)})&=\int_{\mathbb{R}^2} x^{i-1} p(a_1,x;T)\sgn(y-x)y^{j-1} p(a_2,y;T)dxdy,\\
					H(v^{(i)},u^{(j)})&=\int_{\mathbb{R}^2} x^{i-1} p(a_2,x;T)\sgn(y-x)y^{j-1} p(a_1,y;T)dxdy.
				\end{align*}	
				This normalization factor induces the following definition of multiple skew-orthogonal functions.
				\begin{proposition}
					Let's denote
					$m_{i,j}=Q_{\mathbb{R}^2}(u^{(i)},u^{(j)})$, $\tilde{m}_{i,j}=\tilde{Q}_{\mathbb{R}^2}(v^{(i)},v^{(j)})$, $h_{i,j}=H(u^{(i)},v^{(j)})$ and $\psi_{i,j}(x)=x^{j-1}p(a_i,x;T)$. If $n_1+n_2$ is odd, then we could define functions
					\begin{align*}
						&R_{n_1,n_2}(x)=\pf\left[\begin{array}{ccc}
							\left[
							m_{i,j}
							\right]_{i,j=1}^{n_1}&\left[
							h_{i,j}
							\right]_{i=1,\cdots,n_1 \atop j=1,\cdots,n_2}&
							\left[
							\psi_{1,j}(x)
							\right]_{j=1}^{n_1}
							\\
							-\left[
							h_{i,j}
							\right]_{i=1,\cdots,n_2\atop j=1,\cdots,n_1}&
							\left[
							\tilde{m}_{i,j}
							\right]_{i,j=1}^{n_2}
							&
							\left[
							\psi_{2,j}(x)
							\right]_{j=1}^{n_2}
							\\
							-\left[
							\psi_{1,j}(x)
							\right]_{j=1}^{n_1}&-\left[
							\psi_{2,j}(x)
							\right]_{j=1}^{n_2}&0
						\end{array}
						\right],\\
						&R_{n_1+1,n_2}(x)=\pf\left[\begin{array}{ccc}
							\left[m_{i,j}\right]_{i,j=1,\cdots,n_1-1,n_1+1}
							&\left[
							h_{i,j}
							\right]_{i=1,\cdots,n_1-1,n_1+1\atop j=1,\cdots,n_2}&\left[
							\psi_{1,j}(x)
							\right]_{j=1,\cdots,n_1-1,n_1+1}\\
							-[h_{i,j}]_{i=1,\cdots,n_2\atop j=1,\cdots,n_1-1,n_1+1}&\left[\tilde{m}_{i,j}\right]_{i,j=1}^{n_2}&\left[
							\psi_{2,j}(x)
							\right]_{j=1}^{n_2}\\
							-\left[
							\psi_{1,j}(x)
							\right]_{j=1,\cdots,n_1-1,n_1+1}&-\left[\psi_{2,j}(x)\right]_{j=1}^{n_2}&0
						\end{array}
						\right],\\
						&R_{n_1,n_2+1}(x)=\pf\left[\begin{array}{ccc}
							\left[
							m_{i,j}
							\right]_{i,j=1}^{n_1}&\left[
							h_{i,j}
							\right]_{i=1,\cdots,n_1,\atop j=1,\cdots,n_2-1,n_2+1}&\left[
							\psi_{1,j}(x)
							\right]_{j=1}^{n_1}\\
							-\left[
							h_{i,j}
							\right]_{i=1,\cdots,n_2-1,n_2+1\atop j=1,\cdots,n_1}&\left[
							\tilde{m}_{i,j}
							\right]_{i,j=1,\cdots,n_2-1,n_2+1}&\left[
							\psi_{2,j}(x)
							\right]_{j=1,\cdots,n_2-1,n_2+1}
							\\
							-\left[
							\psi_{1,j}(x)
							\right]_{j=1}^{n_1}&-\left[
							\psi_{2,j}(x)
							\right]_{j=1,\cdots,n_2-1,n_2+1}&0
						\end{array}
						\right],
					\end{align*}
					and they satisfy the following multiple skew orthogonality \footnote{Here $\vec{n}=(n_1,n_2)$ and $e_i$ is the unit vector whose $i$th element is $1$, and the others are $0$. For example, $\vec{n}+e_1=(n_1+1,n_2)$.}
					\begin{align*}
						&\int_{\mathbb{R}^2} R_{n_1,n_2}(x)\sgn(y-x)y^{j-1}p(y,a_i;T)dxdy=Z_{\vec{n}+e_i}\delta_{j,n_i+1},&\quad j=1,\cdots,n_i+1,\quad i=1,2,\\
						&\int_{\mathbb{R}^2} R_{\vec{n}+e_i}(x)\sgn(y-x)y^{j-1}p(y,a_i;T)dxdy=Z_{\vec{n}+e_i}\delta_{j,n_i},&\quad j=1,\cdots,n_i+1,\quad i=1,2,
					\end{align*}	
					where $Z_{n_1,n_2}$ is given by \eqref{zn1n2}.								
				\end{proposition}

				Instead, inspired by the generating function in the forms  \eqref{pf3} and \eqref{pf4}, we could introduce a new concept of multiple partial-skew-orthogonal polynomials. Namely, we could define another family of functions $\tilde{R}_{n+1,n_2}$ and $\tilde{R}_{n_1,n_1+1}$ to replace $R_{n_1+1,n_2}$ and $R_{n_1,n_2+1}$ when $n_1+n_2$ is odd. 				
				
				\begin{proposition}
					For $n_1+n_2$ being odd, we could define
					\begin{align*}
						&\tilde{R}_{n_1+1,n_2}(x)=\pf\left[\begin{array}{cccc}
							0&Q_{\mathbb{R}}(u^{(j)})&0&0\\
							-Q_{\mathbb{R}}(u^{(i)})&Q_{\mathbb{R}^2}(u^{(i)},u^{(j)})&H(u^{(i)},v^{(k)})&\psi_{1,i}(x)\\
							0&-H(v^{(l)},u^{(j)})&\tilde{Q}_{\mathbb{R}^2}(v^{(l)},v^{(k)})&\psi_{2,l}(x)\\
							0&-\psi_{1,j}(x)&-\psi_{2,k}(x)&0
						\end{array}
						\right]_{i,j=1,\cdots,n_1+1\atop k,l=1,\cdots,n_2},\\
						&\tilde{R}_{n_1+1,n_2}(x)=\pf\left[\begin{array}{cccc}
							0&0&\tilde{Q}_{\mathbb{R}}(v^{(k)})&0\\
							0&Q_{\mathbb{R}^2}(u^{(i)},u^{(j)})&H(u^{(i)},v^{(k)})&\psi_{1,i}(x)\\
							-\tilde{Q}_{\mathbb{R}}(v^{(l)})&-H(v^{(l)},u^{(j)})&\tilde{Q}_{\mathbb{R}^2}(v^{(k)},v^{(l)})&\psi_{2,l}(x)\\
							0&-\psi_{1,j}(x)&-\psi_{2,k}(x)&0
						\end{array}
						\right]_{i,j=1,\cdots,n_1\atop k,l=1,\cdots,n_2+1}.
					\end{align*}
					Moreover, these functions satisfy the following partial-skew orthogonality
					\begin{align*}
						&\int_{\mathbb{R}^2}R_{\vec{n}+e_1}(x)\sgn(y-x)y^{j-1}p(y,a_i;T)dxdy=Q_{\mathbb{R}}(u^{(j)})Z_{\vec{n}+e_1}\delta_{i,1},\quad 1\leq j\leq n_1+1,\\
						&\int_{\mathbb{R}^2}R_{\vec{n}+e_2}(x)\sgn(y-x)y^{j-1}p(y,a_i;T)dxdy=\tilde{Q}_{\mathbb{R}}(v^{(j)})Z_{\vec{n}+e_1}\delta_{i,2},\quad 1\leq j\leq n_2+1.\\
					\end{align*}
				\end{proposition}				
				As mentioned in \cite{li24}, the advantage of partial-skew-orthogonality mainly lie in the existence of Pfaffian formula for odd order. We hope that these multiple partial-skew-orthogonal polynomials would be useful in the future study of classical integrable systems, including spectral transformations and corresponding Pfaffian $\tau$-functions.			
				
				\section{Generating functions for non-intersecting paths between several sets and intervals}\label{sec5}
				In this section, we consider non-intersecting paths between several sets and intervals. To this end, we consider different sets of vertices $\{\mathbf{u}_i\}_{i=1}^m$ and $\{\mathbf{v}_i\}_{i=1}^n$ and different intervals of vertices $\{I_i\}_{i=1}^n$ and $\{J_i\}_{i=1}^m$. Here we assume that each $\mathbf{u}_i$ has $r_i$ vertices labelled by $\mathbf{u}_i=(u_i^{(1)},\cdots,u_i^{(r_i)})$ and each $\mathbf{v}_j$ has $s_j$ vertices labelled by $\mathbf{v}_j=(v_j^{(1)},\cdots,v_j^{(s_j)})$. Moreover, we should assume that $I_i$ and $J_j$ are both totally ordered sets of vertices.

				\begin{definition}[$T\left(\mathbf{u}_1, \ldots, \mathbf{u}_m; I_1, \ldots I_n\right)$]
					For $m$ different vertex sets $\mathbf{u}_1,\cdots,\mathbf{u}_m$ and $n$ different vertices intervals $I_1,\cdots,I_n$, let's define $T\left(\mathbf{u}_1, \ldots, \mathbf{u}_m, I_1, \ldots I_n\right)$ as direct sums of sets and intervals by inserting $\{I_i\}_{i=1}^n$ into $\mathbf{u}_1\oplus\cdots\oplus \mathbf{u}_m$, and each vertex in $I_i$ precedes those in $I_j$ for $i<j$.
					
				\end{definition}    
				
				For example, if there are two sets of vertices $\mathbf{u}_1,\mathbf{u_2}$ and two intervals $I_1, I_2$, then 
				\begin{align*}
					T(\mathbf{u}_1,\mathbf{u_2};I_1,I_2)=\{\mau_1\oplus\mau_2\oplus I_1\oplus I_2,\mau_1\oplus I_1\oplus \mau_2\oplus I_2,\mau_1\oplus I_1\oplus I_2\oplus\mau_2,\\
					I_1\oplus\mau_1\oplus\mau_2\oplus I_2,I_1\oplus \mau_1\oplus I_2\oplus \mau_2,I_1\oplus I_2\oplus\mau_1\oplus \mau_2\}.
				\end{align*}			
				Moreover, we can define $T_\tau(\mau_1,\cdots,\mau_m;I_1,\cdots,I_n)$ is a specific element in $T(\mau_1,\cdots,\mau_m;I_1,\cdots,I_n)$ with ordering $\tau$. Given this $\tau$, we could define $\hat{T}_\tau(J_1,\cdots,J_m;\mav_1,\cdots,\mav_n)$ by replacing $\mau_i$ by $J_i$ for $i=1,\cdots,m$, and replacing $I_i$ by $\mav_i$ for $i=1,\cdots,n$. For example, if we define
				\begin{align*}
					T_\tau(\mau_1,\mau_2,I_1,I_2)=\mau_1\oplus I_1\oplus \mau_2\oplus I_2,
				\end{align*}
				then we have
				\begin{align*}
					\hat{T}_\tau(J_1,J_2;\mav_1,\mav_2)=J_1\oplus \mav_1\oplus J_2\oplus \mav_2.
				\end{align*}			
				With this setting, we could define non-intersecting paths between several sets and intervals.				
				\begin{definition}[$\mathscr{P}_0(T_\tau\left(\mathbf{u}_1, \cdots, \mathbf{u}_m; I_1, \cdots I_n\right) ; \hat{T}_\tau\left(J_1, \cdots, J_m; \mathbf{v}_1, \cdots \mathbf{v}_n\right))$]
					Given $\tau$ as a specific ordering, let's denote $\mathscr{P}_0(T_\tau\left(\mathbf{u}_1, \cdots, \mathbf{u}_m; I_1, \cdots, I_n\right) ; \hat{T}_\tau\left(J_1, \cdots, J_m; \mathbf{v}_1, \cdots, \mathbf{v}_n\right))$ as all disjoint paths where vertices in $\mathbf{u}_i$ go to $\mathbf{v}_j$ and $J_i$ for any $i=1,\cdots,m$ and $j=1,\cdots,n$, while vertices in $I_j$ go to the rest of $\mathbf{v}_j$. We further assume that the number of paths from $I_j$ to $\mathbf{v}_j$ and those from $\mathbf{u}_i$ to $J_j$ are both even.
				\end{definition}   
				\begin{definition}[Continuity of 1-Factors]
					A 1-factor 
					\begin{align}\label{continuous1factor}
						\pi\in\mathscr{F}\left(u_1^{(1)},\cdots,u_1^{(r_1)},\cdots,u_m^{(1)},\cdots,u_m^{(r_m)},v_{n}^{(s_n)},\cdots,v_n^{(1)},\cdots,v_1^{(s_1)},\cdots,v_1^{(1)}\right)
					\end{align}
					is said to be continuous if it satisfies the following two conditions:
					1. If there is some $u_p^{(k)}$ paired with $v_q^{(l)}$, then $u_p^{(k+1)}$ is paired with $v_q^{(l+1)}$, then $u_p^{(k-1)}$ must be paired with $v_q^{(l-1)}$, until all elements in $\mathbf{u}_p$ or $\mathbf{v}_q$ are fully paired, and the rest in $\mathbf{v_q}$ and $\mathbf{u_p}$ are internally paired separately (see Figure \ref{fig2}.a) ;  2. Under the order $\tau$, if $\mathbf{u}_i$ and $\mathbf{u}_{i+1}$ are adjacent, and $u_i^{(r_i)}$ is paired with $v_q^{(l)}$, then $u_{i+1}^{(1)}$ is paired with $v_q^{(l+1)}$ (see Figure \ref{fig2}.b). This is also true for $\mathbf{v}_i$ and $\mathbf{v}_{i+1}$.	
					
				\end{definition}	
				\begin{figure}[htbp]
					\centering    
					\subfigure[condition 1 of the continuity of 1-factors]{				
						\includegraphics[width=0.60\textwidth]{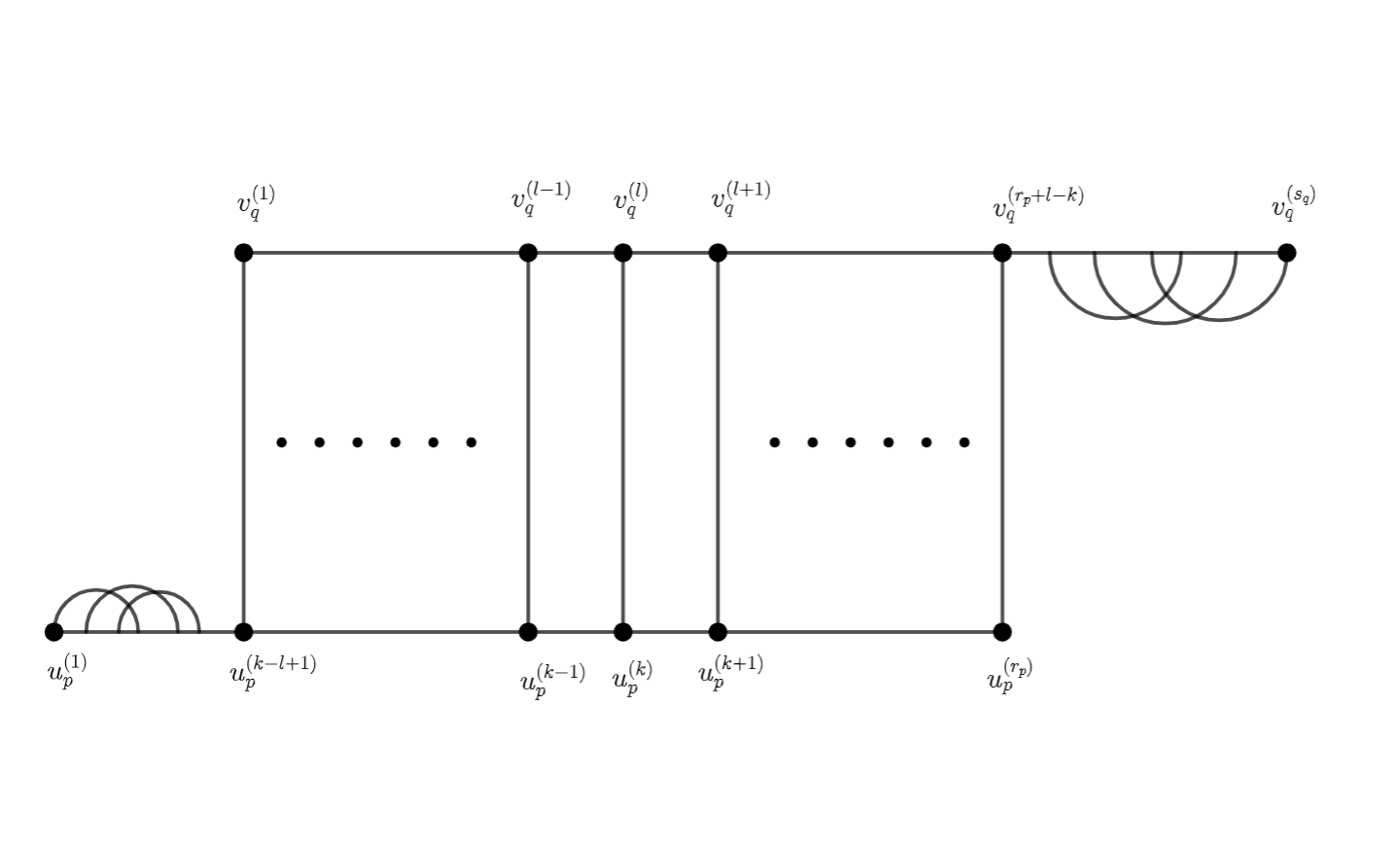}}\\
					\subfigure[condition 2 of the continuity of 1-factors]{				
						\includegraphics[width=0.80\textwidth]{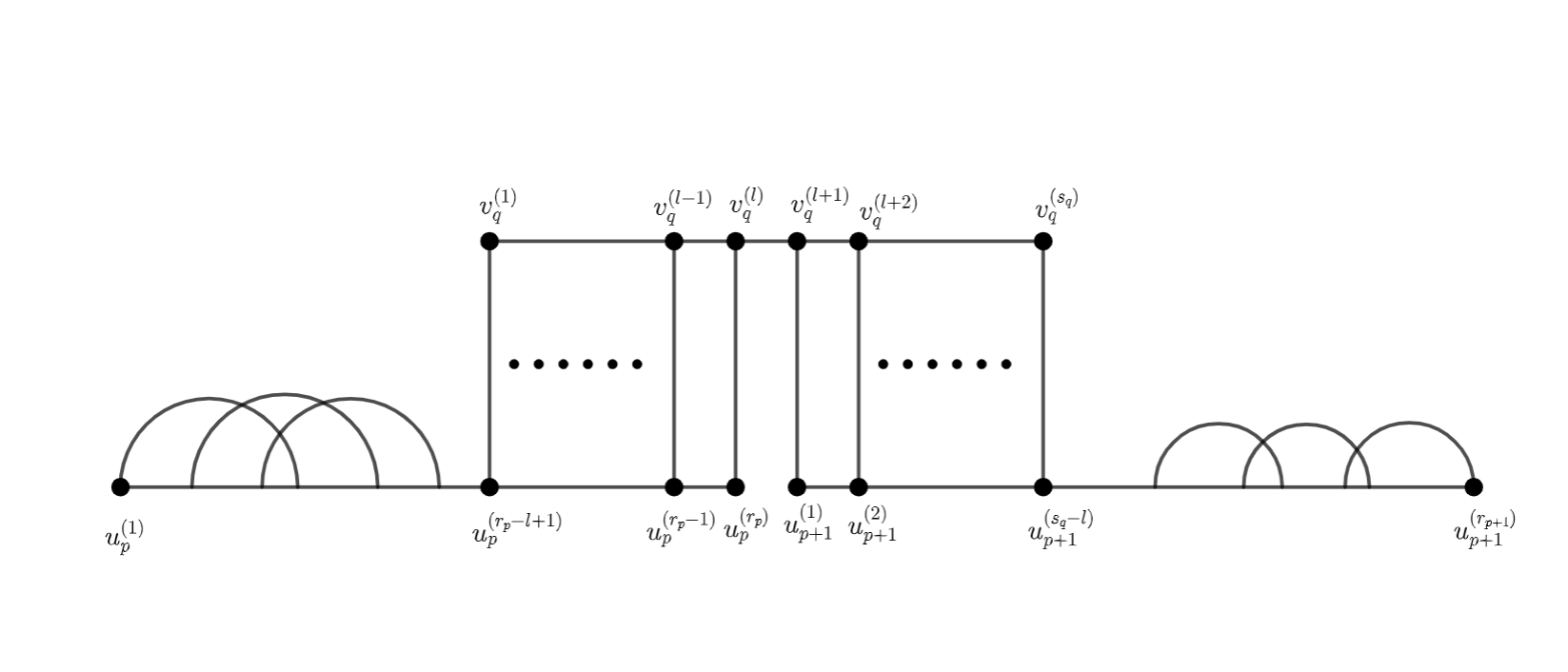}}
					
					\caption{Graphic explanations for the continuity of 1-factors}		
					\label{fig2}
				\end{figure}
				It should be noted that a continuous 1-factor could be divided into three parts. Namely, for any 
				$\pi\in\mathscr{F}\left(u_1^{(1)},\cdots,u_1^{(r_1)},\cdots,u_m^{(1)},\cdots,u_m^{(r_m)},v_{n}^{(s_n)},\cdots,v_n^{(1)},\cdots,v_1^{(s_1)},\cdots,v_1^{(1)}\right)$,
				we have 
				\begin{align*}
					\pi=\sigma \cup \pi_{\sigma_1} \cup \pi_{\sigma_2},
				\end{align*}
				where $\sigma=\bigcup\left(u_p^{(i)}, v_q^{(j)}\right)$, $\pi_{\sigma_1}=\bigcup_{t=1}^m \pi_{\sigma_1}^{(t)}$, $\pi_{\sigma_1}^{(t)} \in \mathscr{F}\left(\mathbf{u}_t \backslash\left(\sigma \cap \mathbf{u}_t\right)\right)$, and
				$\pi_{\sigma_2}=\bigcup_{t=1}^n \pi_{\sigma_2}^{(t)}$ ,$\pi_{\sigma_2}^{(t)} \in \mathscr{F}\left(\mathbf{v}_t \backslash\left(\sigma \cap \mathbf{v}_t\right)\right)$. 
				Here $\tau$ means a continuous pairing between $\mathbf{u}$ and $\mathbf{v}$, and $\pi_{\sigma_1}$ (respectively $\pi_{\sigma_2}$) means a pairing between vertices in $\mathbf{u}$ (respectively in $\mathbf{v}$). We denote all continuous 1-factors as a set $\Lambda_\sigma$.	
				
				Here we remark that $\mau_i$ and $\mau_{i+1}$ could not be merged into one vertex set, since there could be paths from $\mau_i\to I_i$ and from $\mau_{i+1}\to I_{i+1}$. These paths are different.

				\begin{definition}[$D$-separated]
					Two ordered sets of vertices $T_\tau\left(\mathbf{u}_1, \mathbf{u}_2, \ldots \mathbf{u}_m, J_1, J_2, \ldots, J_n\right)$ and $\hat{T}_\tau\left(I_1, I_2, \ldots I_m, \mathbf{v}_1, \mathbf{v}_2, \ldots, \mathbf{v}_n\right)$ are $D$-separated if any path from $\mathbf{u}_i$ to $I_i$ does not intersect with any path starting from $\mathbf{u}_k$ $(k \neq i)$, and any path from $J_j$ to $\mathbf{v}_j$ does not intersect with any path to $\mathbf{v}_l$ $(l \neq j)$. Moreover, and any path from $\mathbf{u}_i$ to $I_i$ does not intersect with any path from $J_j$ to $\mathbf{v}_j$ for all $i=1,\cdots,m$ and $j=1,\cdots,n$. 				\end{definition}

				\begin{theorem}					
					Let $\mathbf{u}_i=\left(u^{(1)}_i, u^{(2)}_i, \ldots, u^{(r_i)}_i\right)$ and $\mathbf{v}_j=\left(v^{(1)}_j, \ldots, v^{(s_j)}_j\right)$ $(1 \le i \le m,1 \le j \le n)$ be sequences of vertices in the acyclic directed graph $D$, and assume that $\sum_{i=1}^m r_i+\sum_{j=1}^n s_j$ is even.
					If for any $1 \leq i \leq m, 1 \leq j \leq n$, $I_j$, $J_i$ are totally ordered subsets of $V$, such that $T_\tau\left(\mathbf{u}_1, \mathbf{u}_2, \ldots, \mathbf{u}_m, I_1, I_2, \ldots I_n\right)$ and $\hat{T}_\tau\left(J_1, J_2, \ldots, J_m, \mathbf{v}_1, \mathbf{v}_2, \ldots \mathbf{v}_n\right)$ are $D$-compatible and $D$-separated.
					Then
					\begin{align} \label{pf2}
						{\rm{GF }}\left[\mathscr{P}_0\left( T_\tau\left(\mathbf{u}_1, \mathbf{u}_2, \ldots, \mathbf{u}_m, I_1, I_2, \ldots I_n\right); \hat{T}_\tau\left(J_1, J_2, \ldots, J_m, \mathbf{v}_1, \mathbf{v}_2, \ldots \mathbf{v}_n\right) \right)\right] = \pf\left[ \begin{array}{cc}
							{Q} & {H} \\
							-{H}^T & {\tilde{{Q}}}
						\end{array} \right],
					\end{align}
					where
					\begin{align*}
						&{{ Q}}=\text{diag}\left(
						Q_{J_1}(u_1^{(k)},u_1^{(l)})_{1\leq k,l\leq r_1},\cdots,Q_{J_m}(u_m^{(k)},u_m^{(l)})_{1\leq k,l\leq r_m}
						\right),
						\\
						&{{\tilde Q }}=\text{diag}\left(\tilde{Q}_{I_n}(v_n^{(k)},v_n^{(l)})_{1\leq k,l\leq s_n},\cdots,\tilde{Q}_{I_1}(v_1^{(k)},v_1^{(l)})_{1\leq k,l\leq s_1} \right),
						\\
						&{{H}}=\left(
						H_{\mau_i,\mav_j}
						\right)_{i=1,\cdots,m,j=n,\cdots,1},\quad H_{\mau_i,\mav_j}=\left(h(u_i^{(k)},v_j^{(s_j+1-l)})\right)_{k=1,\cdots,r_i,l=1,\cdots,s_j}.
					\end{align*}

				\end{theorem}
				
				\begin{proof}
					First, we know that the right hand side of equation \eqref{pf2} could be expanded as
					
					\begin{align*}
						\sum\limits_\pi  {{\sgn}} (\pi )
						\left( \prod\limits_{(u_k^{(i)},u_k^{(j)}) \in \pi } Q_{I_k}(u_k^{(i)},u_k^{(j)}) \right)\left( \prod\limits_{(u_k^{(i)},v_l^{(j)}) \in \pi } h(u_k^{(i)},v_l^{(j)}) \right)\left(\prod\limits_{(v_l^{(j)},v_l^{(i)}) \in \pi } {\tilde Q}_{J_l} (v_l^{(j)},v_l^{(i)}) \right)
					\end{align*}
					according to the formula \eqref{expansionformula}, where the summation is over all continuous 1-factors \eqref{continuous1factor}. Moreover, we could recognize this formula as a product of path weights and rewrite it as 
					\begin{align}
						\pf\left[\begin{array}{cc}
							Q & H \\
							-H^T & \tilde{Q}
						\end{array}\right]=\sum_{\left(\pi, \mathbf{P}_1, \ldots, \mathbf{P}_m, \mathbf{Q}_1, \ldots, \mathbf{Q}_n\right)} \operatorname{sgn}(\pi) \omega\left(\pi, \mathbf{P}_1, \ldots, \mathbf{P}_m, \mathbf{Q}_1, \ldots, \mathbf{Q}_n\right),
					\end{align}
					where $\mathbf{P}_k$ and $\mathbf{Q}_k$ are a set of paths, 
					$\mathbf{P}_k=\left(P_k^{(1)}, P_k^{(2)}, \ldots, P_k^{\left(r_k\right)}\right)$, $\mathbf{Q}_l=\left(Q_l^{(1)}, Q_l^{(2)}, \ldots, Q_l^{\left(s_l\right)}\right)$,
					and $\left(\pi, \mathbf{P}_1, \ldots, \mathbf{P}_m, \mathbf{Q}_1, \ldots, \mathbf{Q}_n\right)$ is an $(r+s+1)$-tuple configuration such that
					
					\begin{itemize}
						\item If $\left(u_k^{(i)}, v_l^{(j)}\right) \in \pi$,then  $P_k^{(i)}=Q_l^{\left(j\right)} \in \mathscr{P}\left(u_k^{(i)}, v_l^{(j)}\right)$; 
						\item If $\left(u_k^{(i)}, u_k^{(j)}\right) \in \pi$,then $P_k^{(i)} \in \mathscr{P}\left(u_k^{(i)} ; I_k\right)$,$P_k^{(j)} \in \mathscr{P}\left(u_k^{(j)} ; I_k\right)$,while $P_k^{(i)}$ and $P_k^{(j)}$ don't intersect; 
						\item If $\left(v_l^{(j)}, v_l^{(i)}\right) \in \pi$,then $Q_l^{(j)} \in \mathscr{P}\left(J_l ; v_l^{(j)}\right)$,$Q_l^{(i)} \in \mathscr{P}\left(J_l ; v_l^{(i)}\right)$,while $Q_l^{(j)}$ and $Q_l^{(i)}$ don't intersect.
					\end{itemize}
					Then, we seek for a sign-reversing involution to offset the intersecting paths.
					Since the graph is $D$-separated, we know that if paths intersect, then the intersecting point belongs to $P_k^{(i)}\cap P_l^{(j)}$ or $Q_k^{(i)}\cap Q_l^{(j)}$ for some $k,l$ and $i,j$. Similar to the proof in Theorem \ref{thm4.3}, we find the smallest intersection point $w^*$\footnote{In this case, we find the smallest intersection point starting from paths in $\mathbf{P}_1$. If there is an intersection point in $\mathbf{P}_1$, then we find it from paths starting from $u_1^{(1)}$ and the rest is the same with the case in Theorem \ref{thm4.3}. If there isn't any intersection point in $\mathbf{P}_1$, then we find it in $\mathbf{P}_2$, and so on.}, and assume that $P_k^{(i)}$ and $P_l^{(j)}$ are the two whose subcripts (priority) $k,l$ and supercripts (followed) $i,j$ are the smallest. Thus, we could construct an involution
					\begin{align*}
						\phi: \left(\pi, \mathbf{P}_1, \ldots, \mathbf{P}_m, \mathbf{Q}_1, \ldots, \mathbf{Q}_n\right)\to \left(\pi, \bar{\mathbf{P}}_1, \ldots, \bar{\mathbf{P}}_m, \bar{\mathbf{Q}}_1, \ldots, \bar{\mathbf{Q}}_n\right)
					\end{align*}
					where 
					\begin{align*}
						\bar{P}_k^{(i)}=P_k^{(i)}(\to w^*)P_l^{(j)}(w^*\to),\quad \bar{P}_l^{(j)}=P_l^{(j)}(\to w^*)P_k^{(i)}(w^*\to),
					\end{align*}
					and the rest paths remain invariant. Using the method shown in Theorem \ref{thm4.3}, we can verify that this map is a sign-reversing involution by the property of $D$-separated graph.
					
					By using the sign-reversing involutions to cancel the intersection terms, we find that only configurations containing continuous 1-factors are left. As we divide the continuous 1-factor into 3 parts,   \eqref{pf2} could be rewritten as the non-intersecting part of
					
					\begin{align}\label{simplifiedd}
						\begin{aligned}
						\sum_{\sigma}\sum_{\pi\in\Lambda_\sigma }  \operatorname{sgn}(\pi)
							&\left(\prod_{k=1}^m \prod_{l=1}^n \prod_{\left(u_k^{(i)}, v_l^{(j)}\right) \in \sigma} h\left(u_k^{(i)}, v_l^{(j)}\right)\right)\\
							&\times\left(\prod_{k=1}^m \prod_{\left(u_k^{(i)}, u_k^{(j)}\right) \in \pi_{\sigma_1}} Q_{I_k}\left(u_k^{(i)}, u_k^{(j)}\right)\right)\left(\prod_{l=1}^n \prod_{\left(v_l^{(j)}, v_l^{(i)}\right) \in \pi_{\sigma_2}} \tilde{Q}_{J_l}\left(v_l^{(j)}, v_l^{(i)}\right)\right)
						\end{aligned}
					\end{align}
					\begin{figure}[htbp]
						\centering
						\includegraphics[width=10cm]{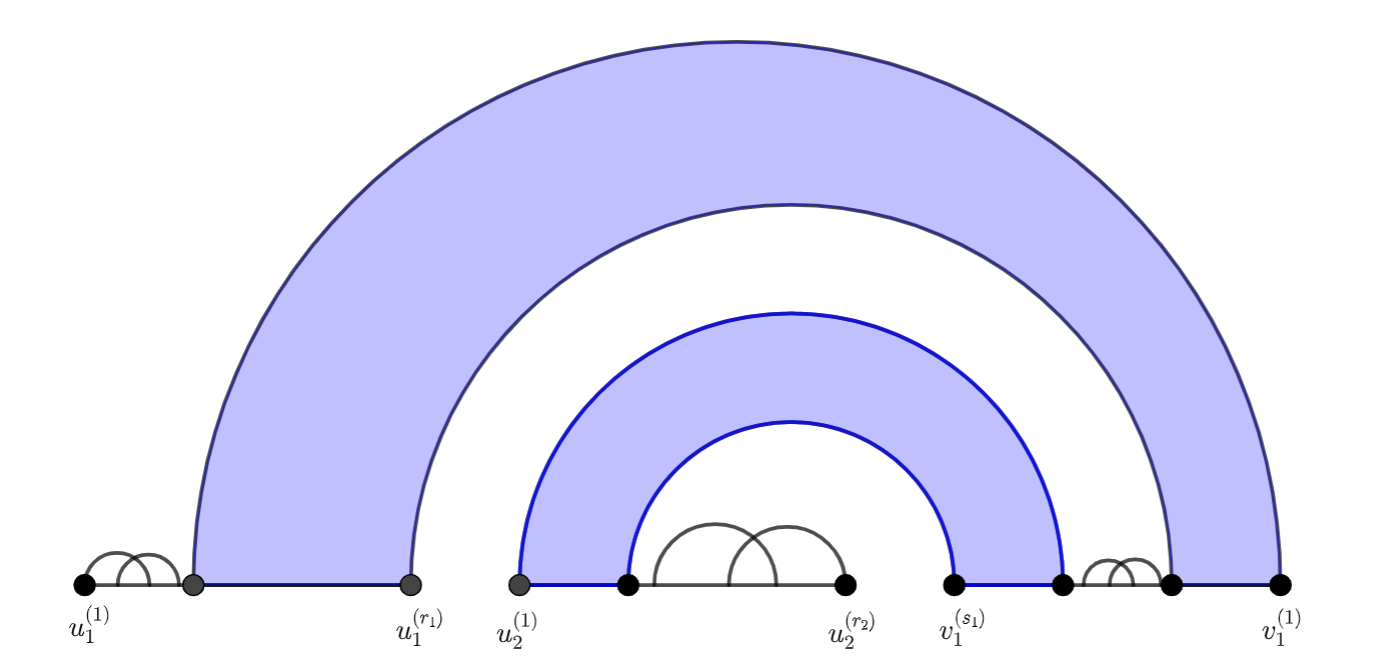}
						\caption{An explanation of non-intersecting pairing between $\mathbf{u}_1\oplus I_1\oplus\mathbf{u}_2\to J_1\oplus\mathbf{v}_1\oplus J_2$.} 
					\end{figure}
					
					Moreover, by realizing
					\begin{align*}\sum\limits_{\pi  \in {\Lambda _\sigma }} {{\mathop{\rm sgn}} } (\pi ) = \left( {\sum {{\mathop{\rm sgn}} } \left( {{\pi_{\sigma_1^{(1)}}}} \right)} \right)\cdots\left(
						\sum\sgn\left(
						\pi_{\sigma_1^{(m)}}
						\right)
						\right)
						\left( {\sum {{\mathop{\rm sgn}} } \left( {{\pi_{\sigma_2^{(1)}}}} \right)} \right) \cdots \left( {\sum {{\mathop{\rm sgn}} } \left( {{\pi_{\sigma_2^{(n)}}}} \right)} \right) = 1,
					\end{align*}
					we know that \eqref{simplifiedd} represents $\gf\left[\mathscr{P}_0\left(T_\tau\left(\mathbf{u}_1, \ldots \mathbf{u}_m, J_1,  \ldots, J_n\right); \hat{T}_\tau\left(I_1, \ldots I_m, \mathbf{v}_1,  \ldots, \mathbf{v}_n\right)\right)\right]$, and there are even paths from $\mathbf{u}_k\to I_k$ and even paths from $J_l\to\mathbf{v}_l$.

				\end{proof}

				\section*{Acknowledgement}
				We would like to thank the referee for useful comments. This work is partially funded by grants (NSFC12101432, NSFC12175155).

			\end{document}